\documentclass{amsart}
	\usepackage{aliascnt}
	\usepackage[colorlinks=true, linkcolor=black, citecolor=black, urlcolor=black]{hyperref}
	\usepackage{amssymb}	
	\usepackage{amsmath}
	\usepackage{amsthm}
	\usepackage{MnSymbol}

	\newaliascnt{lemma}{thm}
	\newtheorem{lemma}[lemma]{Lemma}  
	\aliascntresetthe{lemma}

	\newaliascnt{prop}{thm}
	\newtheorem{prop}[prop]{Proposition} 
	\aliascntresetthe{prop}

	\newaliascnt{cor}{thm}
	
	\aliascntresetthe{cor}

	\theoremstyle{remark}

	\newaliascnt{rem}{thm}
	\newtheorem{rem}[rem]{Remark}
	\aliascntresetthe{rem}

	\theoremstyle{definition}

	\newaliascnt{exm}{thm}
	
	\aliascntresetthe{exm}

	\newaliascnt{notn}{thm}
	
	\aliascntresetthe{notn}

	\newaliascnt{defn}{thm}
	\newtheorem{defn}[defn]{Definition}
	\aliascntresetthe{defn}

	\usepackage{layout}
	\newcommand{\K}{\mathbb{K}}
	\newcommand{\rk}{\operatorname{rk}}
	\newcommand{\rmax}{\operatorname{r_{max}}}
	
	\newcommand{\con}{\operatorname{\partial}}
	\newcommand{\Sy}{\operatorname{Sym}}
	\newcommand{\Ker}{\operatorname{Ker}}
	\newcommand{\Ps}{\mathbb{P}}
	\newcommand{\ins}{\minushookup}
	
	\newcommand{\Span}[1]{\left\langle\,#1\,\right\rangle}
	
\begin{document}

\title[Every ternary quintic is a sum of ten fifth powers]{Every ternary quintic is a sum\\of ten fifth powers}
\author{Alessandro De Paris}
\address{Dipartimento di Matematica e Applicazioni ``Renato Caccioppoli'',\newline\indent Universit\`a di Napoli Fe\-de\-ri\-co II (Italy)}
\email{deparis@unina.it}

\begin{abstract}
To our knowledge at the time of writing, the maximum Waring rank for the set of \emph{all} ternary forms of degree $d$ (with coefficients in an algebraically closed field of characteristic zero) is known only for $d\le 4$. The best upper bound that is known for $d=5$ is twelve, and in this work we lower it to~ten.
\end{abstract}

\maketitle



\section{Introduction}
The target of the present paper is the Waring problem for the set of \emph{all} forms of fixed degree and in a fixed number of variables, with coefficients in an algebraically closed field of zero characteristic. This problem is part of a body of questions which are under renewed interest, because of the recent discovery of new applications (see the book \cite{L}). General information can be found in nearly everyone of the several articles that have recently been written in this topic (e.g., in \cite{T}).

The best upper bound on the Waring rank of an arbitrary form $f$ of degree $d$ and in $n$ variables, to our knowledge at the time of writing, is given by \cite[Corollary~9]{BT}: apart from a few exceptional pairs $(n,d)$, we have
\begin{equation}\label{RBT}
\rk f\le 2\left\lceil\frac 1n\binom{n+d-1}{n-1}\right\rceil\;.
\end{equation}
That result is based on the Alexander-Hirschowitz theorem (see \cite{AH}), which gives the rank of a \emph{general} $f$, for each fixed pair $(n,d)$. One would like to determine the sharp bound $\rmax(n,d)$.  A lower bound for $\rmax(n,d)$ is given, of course, by the rank of general forms, which is half of the above upper bound. In the case of ternary forms, $n=3$, the rank of monomials (see \cite[Proposition~3.1]{CCG}) gives a better lower bound:
\[
\rmax(3,d)\ge\left\lceil\frac{d^2+2d}{4}\right\rceil\;.
\]
Lower bounds for polynomials of special type are also intensively studied (see, e.g., \cite{T}, \cite{DT}). Some upper bounds that do not rely on the Alexander-Hirschowitz theorem turn out to be better than \eqref{RBT} in low degree (see \cite[Corollary~6]{J}, \cite[Propositions~3.9 and~4.2]{BD}). In this situation, to study the unknown case with least $(n,d)$ seems a reasonable way to seek for inspiration. In this paper we find \[\rmax(3,5)\le 10\;.\] Note that, according to \cite[Proposition~3.1]{CCG}, there exist degree five monomials in three variables whose rank is nine, and that when $n=3$ and $2\le d\le 4$, as well as when $n\ge 4$ and $d\ge 2$, $\rmax(n,d)$ is not reached by monomials. On the other hand, we have some reasons to believe that to find a rank ten ternary quintic might be harder than one would expect and, at this point, we can not exclude that is impossible (\footnote{But the needed example is part of an article in preparation by Jaros{\l}aw Buczy\'{n}ski and Zach Teitler (personal communication).}).

For the introductory purposes, here we quickly recapitulate the content of \cite{D}, where a way to determine $\rmax(3,4)$ is presented, and subsequently outline the enhancements we are obtaining here.

All vector spaces are understood over a fixed algebraically closed field $\K$ of zero characteristic. We fix standard graded rings $S_\bullet=\Sy^\bullet S_1$, $S^\bullet=\Sy^\bullet S^1$ and a dual pairing between $S^1$ and $S_1$. The dual pairing naturally extends to $S^\bullet$, $S_\bullet$, giving rise to the \emph{apolarity pairing} (for details, see \cite[Introduction]{D}). The \emph{contraction operation}
\[
l\ins v\in S_\bullet\;,\qquad l\in S^\bullet,f\in S_\bullet\;,
\]
is easily described in terms of apolarity and, on the other hand, simply amounts to constant coefficients partial derivation, when dual bases
\[
x^0,\ldots ,x^n\in S^1\;,\qquad x_0,\ldots ,x_n\in S_1
\]
are fixed ($x^i\ins f\left(x_0,\ldots ,x_n\right)=\partial f/\partial x_i$). The sign $\perp$ will refer to orthogonality with respect to the apolarity pairing \emph{in fixed degree}; we shall not use it to denote apolar ideals. A projective space $\Ps V$ is understood as the set of all one-dimensional subspaces $\Span{v}$ of the vector space $V$. Given $f\in S_d$, its (Waring) rank is denoted by $\rk f$.

Let us start by recalling the situation of \cite[Lemmas~2.1 and~2.4]{D}, which deal with binary forms. From the viewpoint of rank determination, these elementary objects exhibit a nontrivial behavior which, nevertheless, is well-understood in its general lines (see, e.g., \cite{S}, \cite{IK}, \cite{CS}). So, let us assume $\dim S_1=2$ for the moment.

Let $W\subset S_4$ be a subspace of dimension three. We also put the hypothesis that, for some linearly independent $x_0,x_1\in S_1$, $W$ contains ${x_0\,}^4,{x_1\,}^4$ (as in \cite[Lemma~2.1]{D}) or ${x_0\,}^4,{x_0\,}^3x_1$ (as in \cite[Lemma~2.4]{D}). To take a geometric view, we consider the plane $\Ps W$ in the four-space $\Ps S_4$. We look at $\Ps S_d$ as the ambient of a canonical rational normal curve $C_d$, through a Veronese embedding $\nu_d:\Ps S_1\twoheadrightarrow C_d\hookrightarrow\Ps S_d$, simply given by $\nu_d\left(\Span{v}\right):=\Span{v^d}$. In the case $W\supset\left\{{x_0\,}^4,{x_1\,}^4\right\}$, the plane $\Ps W$ meets $C_4$ in (at least) two distinct points $\Span{{x_0\,}^4}$, $\Span{{x_1\,}^4}$; in the case $W\supset\left\{{x_0\,}^4,{x_0\,}^3x_1\right\}$, $\Ps W$ is tangent to $C_4$ in $\Span{{x_0\,}^4}$. This gives a line $\Ps L\subset\Ps W$, that is secant ($L=\Span{{x_0\,}^4,{x_1\,}^4}$) or tangent ($L=\Span{{x_0\,}^4,{x_0\,}^3x_1}$) to $C_4$. To regard $A:=\Ps W\setminus\Ps L$ as an affine plane with line at infinity $\Ps L$ is also convenient. The mentioned lemmas give information on the rank stratification in $W$; namely, they describe the loci
\[
R:=\left\{\Span f\in A: \rk f\ne 3\right\}\;,\qquad
R':=\left\{\Span f\in A: \rk f=4\right\}
\]
(note that $R\setminus R'$ is precisely the set of $\Span{f}\in A$ with $\rk f\le 2$). For instance, in the secant case (\cite[Lemma~4.1]{D}) we have one of the following alternatives \ref{c11}, \ref{c12}, \ref{c2}:
\begin{enumerate}
\item\label{c1} $R'$ consists of at most two points and
	\begin{enumerate}
	\item\label{c11} $R\ne\emptyset$ is an affine conic with points at infinity exactly $\Span{{x_0}^4}$, $\Span{{x_1}^4}$, and when $R$ possesses a singular point $\Span{x}$ we have $\rk x=1$ and $R'=\emptyset$; or
	\item\label{c12} $R\ne\emptyset$ is an affine line with point at infinity different from $\Span{{x_0}^4}$, $\Span{{x_1}^4}$; or
	\end{enumerate}
\item\label{c2} $R=R'\ne\emptyset$ is an affine line with point at infinity either $\Span{{x_0}^4}$ or $\Span{{x_1}^4}$, and, more precisely, $R=R'=A\cap\Ps{\Span{{x_0}^4,{x_0}^3x_1}}$ in the first case, $R=R'=A\cap\Ps{\Span{x_0{x_1}^3,{x_1}^4}}$ in the other.
\end{enumerate}
In the tangent case the alternatives are similar: see \cite[Lemma~2.4]{D}. These results are quite elementary, and may be proved in several different ways. A geometric argument of a common kind is outlined in \cite[discussion after Lemma~2.1]{D}. It is based on the projection of $A$ from $L$, $\Ps S_d\setminus\Ps L\to\Ps (S_d/L)$. Indeed, we get a point in a plane, whose position with respect to the projection of $C_4$, which is a conic, determines the alternative that occurs. \hyperref[c2]{Case~\ref{c2}} is perhaps the worst, in view of subsequent applications.

Now, let us consider ternary forms, and so assume $\dim S_1=3$, $f\in S_4$. According to \cite[Proposition~4.1]{D}, there exist distinct $\Span{x^0},\Span{x^1},\Span{l}\in\Ps S^1$ such that $x^0x^1l\ins f=0$. We look again at $\Ps S_4$ as the ambient of a Veronese embedding $\nu_4:\Ps S_1\to\Ps S_4$, $\Span{v}\mapsto\Span{v^4}$. Each of $\Ps\Span{x^0}^\perp$, $\Ps\Span{x^1}^\perp$, $\Ps\Span{l}^\perp$, is a line in $\Ps S_1$ that is mapped by $\nu_4$ into a rational normal quartic in a space of (essentially) binary forms, say $\Ps V_0$, $\Ps V_1$, $\Ps V_2$ ($V_0=\Sy^4\Span{x^0}^\perp$, etc.). Since $x^0x^1l\ins f=0$, we have $f\in V_0+V_1+V_2$. This allows to decompose $f$ (in several ways) as a sum of binary forms, each belonging to a subspace of the form $W$ described before; then, one can exploit the information provided by the lemmas to bound the rank of $f$: see the proofs of \cite[Propositions~3.1 and~5.1]{D}.

In the present work we pursue the same idea. First of all, \cite[Proposition~4.1]{D} has already been generalized (see \cite[Proposition~2.7]{BD}), and this allows us to decompose every ternary quintic into a sum of four binary quintics. Hence, by suitably generalizing \cite[Lemmas~2.1 and~2.4]{D}, we are lead to find upper bounds on the rank. The main difficulty is that a case-by-case strategy like that of \cite[Propositions~3.1 and~5.1]{D} becomes considerably more complicated, and that is why in this paper we also perform some nontrivial reductions.

Although the help of a geometric picture is invaluable to drive arguments, we also need to write down some related equations (in particular, this simplifies the extension of the analysis that was performed in the proof of \cite[Proposition~2.3]{D}). To this end, in \autoref{ALBF} we skip to a purely algebraic setting, we present an extension of \cite[Lemmas~2.1 and~2.4]{D}, and also discuss an additional condition (which has already appeared in \cite{BD}) that allows to avoid \hyperref[c2]{Case~\ref{c2}} mentioned before. Other preparatory results are set up in \autoref{Lines}. They can be regarded as complements to \cite[Proposition~2.7]{BD}, but limited to the case of quintics. The final bound is stated in \autoref{Bound}.

\section{Standing Notation}

As anticipated in the introduction, we work over an algebraically closed field $\K$ of characteristic zero, $S^\bullet, S_\bullet$ denote dually paired, standard graded rings, and contraction is denoted by $\ins$. A projective space $\Ps V$ is understood as the set of all one-dimensional subspaces $\Span{v}$ of the vector space $V$. One may set $\Ps^n:=\Ps\K^{n+1}$, and allow the classical notation $\left[a_0,\ldots ,a_n\right]$ for $\Span{\left(a_0,\ldots ,a_n\right)}\in\Ps^n$ (\footnote{In the usual formalism, $\left[a_0,\ldots ,a_n\right]=\Span{\left(a_0,\ldots ,a_n\right)}\setminus\{(0,\ldots ,0)\}$; but of course this causes no technical problems in the present setting.}). Given $f\in S_d$, its (Waring) rank will be denoted by $\rk f$.

Given $f\in S_{d+\delta}$, we define the \emph{partial polarization map}, $f_{\delta,d}:S^\delta\to S_d$, by setting $f_{d,\delta}(t):=t\ins f$. Given $x\in S^\bullet$, we shall denote by $\con_x:S_\bullet\to S_\bullet$ the contraction by $x$ operator $f\mapsto x\ins f$.

In a few cases, to interpret elements of $S^d$ as homogeneous polynomial functions on $S_1$ will be convenient. In view of that, sometimes we shall use the shortcut
\begin{equation}\label{Ev}
p(v):=\frac1{d!}p\ins v^d=\frac1{d!}\con_p\left(v^d\right)\;,\qquad p\in S^d,v\in S_1
\end{equation}
(when $d=1$ we have $p(v)=p\ins v=\con_p(v)$).

In order to efficiently manipulate parameterizations, we fix a further standard graded ring $\K\left[t^0,t^1\right]$ in two indeterminates $t^0,t^1$, and a bigraded ring
\[
\mathbf{S}:=\K\left[t^0,t^1\right]\otimes S_\bullet\;,
\]
with the bigrading being given by
\[
\mathbf{S}^\delta_d:=\K\left[t^0,t^1\right]_\delta\otimes S_d\;.
\]
The contraction operation can be extended on $\mathbf{S}$, by letting $S^\bullet$ act trivially on $\K\left[t^0,t^1\right]$. That is, we denote again by $\con_x$ the operator
\[
\operatorname{id}\otimes\con_x:\K\left[t^0,t^1\right]\otimes S_\bullet\to\K\left[t^0,t^1\right]\otimes S_\bullet
\]
for all $x\in S^\bullet$, and moreover, for all $\mathbf{f}\in\mathbf{S}$ the notation $x\ins\mathbf{f}$ will stand for $\con_x(\mathbf{f})$. Informally speaking: $t^0,t^1$ behave as constants with respect to the differential operators given by the contraction with any $x\in S^\bullet$. Also the shortcut~\eqref{Ev} can naturally be extended to any $\mathbf{v}\in\K\left[t^0,t^1\right]\otimes S_1$:
\[
p(\mathbf{v}):=\frac1{d!}p\ins \mathbf{v}^d\;.
\]
On the other hand, every element $\mathbf{f}\in\mathbf{S}$ can be evaluated in the obvious way at $(\lambda,\mu)\in\K^2$ (\footnote{That is, the evaluation homomorphism $\mathbf{ev}_{(\lambda,\mu)}:\mathbf{S}\to S_\bullet$ is simply $\operatorname{ev}_{(\lambda,\mu)}\otimes\operatorname{id}$, with $\operatorname{ev}_{(\lambda,\mu)}:\K\left[t^0,t^1\right]\to\K$ being the ordinary evaluation. If a basis $x_0,\ldots, x_n$ of $S_1$ is fixed, it amounts to the ordinary substitution $t^0\mapsto\lambda$, $t^1\mapsto\mu$ into polynomials in $t^0,t^1,x_0,\ldots ,x_n$.}), or more generally at $(\lambda,\mu)\in K^2$, with $K$ being any commutative $\K$-algebra. We shall use the notation $\mathbf{f}\restriction_{(\lambda,\mu)}$ for the evaluation of $\mathbf{f}\in\mathbf{S}$ at $(\lambda,\mu)\in K^2$, which lies in $K\otimes S_\bullet$ (in particular, it lies in $S_\bullet$ when $K=\K$ and in $\mathbf{S}$ when $K=\K\left[t^0,t^1\right]$).

\begin{rem}\label{Extend}
Every $\mathbf{f}\in\mathbf{S}^\delta_d$ gives rise to a map
\begin{equation}\label{Param}
U\to\Ps S_d\;,\qquad U\subseteq\Ps^1,\qquad\left[\lambda,\mu\right]\mapsto\Span{\mathbf{f}\restriction_{(\lambda,\mu)}}\;,
\end{equation}
that parameterizes a (rational) curve in $\Ps S_d$. Suppose that $\mathbf{f}$ admits a (homogeneous) divisor $a\in\K\left[t^0,t^1\right]\subseteq\mathbf{S}$, of positive degree. Then the parameterization \eqref{Param} is undefined at the zeroes of $a$, and $\mathbf{f}/a$ gives an extended parameterization. If $a$ is a divisor of greatest degree (among those in $\K\left[t^0,t^1\right]$), then the extended parameterization is defined on the whole of $\Ps^1$.
\end{rem}

Finally, we explicitly recall an elementary fact which holds, more generally, when the coefficients are in a field of characteristic not dividing $d$.

\begin{rem}\label{Elem}
Let $d\ge 2$ and consider a linear combination of two $d$-th powers of linear forms, with both nonzero coefficients. If it is again a $d$-th power, then the two powers must be proportional (from a geometric viewpoint: rational normal curves of degree $d\ge 2$ admit no trisecant lines).
\end{rem}

The following more general fact is also well-known.

\begin{rem}\label{Elem2}
If $\Span{v_0},\ldots,\Span{v_d}\in\Ps S_1$ are distinct, then $\Span{{v_0}^d},\ldots,\Span{{v_d}^d}\in\Ps S_d$ are linearly independent.
\end{rem}

\section{Ancillary Lemmas on Binary Forms}\label{ALBF}

Throughout this section we assume $\dim S_1=2$. To generalize the results on binary forms we outlined in the introduction, let us consider a $(d-2)$-dimensional projective subspace $\Ps W$ in $\Ps S_d$, with $d\ge 3$. Let $\nu_d:\Ps S_1\to \Ps S_d$ denote again the Veronese embedding given by $\nu_d\left(\Span{v}\right):=\Span{v^d}$, and set $C_d:=\nu_d\left(\Ps S_1\right)$. If $\Span{l}\in\Ps S^1$ and we set $\Span{v}:=\Span{l}^\perp$, then $\Span{v^d}\in S_d$ is the kernel of the restriction $S_d\to S_{d-1}$ of the contraction operator~$\con_l$. More generally, if $p\in S^\delta$ and $L:=S_d\cap\operatorname{Ker}\con_p$, then $\Ps L$ can be regarded as the subspace of $\Ps S_d$ spanned by $\nu_d(Z)$, with $Z$ in $\Ps S_1$ being given by $p=0$ and `counted with multiplicities' (\footnote{We shall not strictly need that statement, which serves only to provide a geometric insight; in any case, it could easily be made precise. For instance, using elementary scheme theory, $Z$ would be the subscheme given by $\operatorname{Proj}\left(S^\bullet/(p)\right)\hookrightarrow\operatorname{Proj}S^\bullet$, and the span would be given by the intersection of all linear subschemes that contains $\nu_d(Z)$ scheme-theoretically. One might also easily avoid schemes.}). For instance, in the situation described in the introduction, if $\Span{l^0}:=\Span{x_0}^\perp,\Span{l^1}:=\Span{x_1}^\perp$, then in the secant case we have $L=S_4\cap\operatorname{Ker}\con_{l^0l^1}$, and in the tangent case we have $L=S_4\cap\operatorname{Ker}\con_{{l^0}^2}$. Therefore, the projection $\Ps S_d\setminus\Ps L\to\Ps (S_d/L)$ can be substituted by $\Ps\left(\pi_p\right)$, with $\pi_p$ being the restriction $S_d\to S_2$ of $\con_p$. The hypothesis that $\Ps W$ meets the curve $C_d$ in a group $Z$ of $d-2$ points (counted with multiplicities), is simply replaced by $W:=S_d\cap\con_p^{-1}\left(\Span{q}\right)$, with $\Span{p}\in\Ps S^{d-2}$ and $\Span{q}\in\Ps S_2$. With those assumptions, the projection of $C_d$ is simply replaced by $C_2$.

Since the above described situation will often occur in the present paper, to set up some related notation will ease the exposition.

\begin{defn}\label{WLA}
When $\dim S=2$ and $p\in S^d$, $q\in S_e$ are nonzero (binary) forms, we set
\[
W_{p,q}:=S_{d+e}\cap\con_p^{-1}\left(\Span{q}\right)\;,\quad L_{p,e}:=S_{d+e}\cap\operatorname{Ker}\con_p\subset W_{p,q}\;,\quad A_{p,q}:=\Ps W_{p,q}\setminus\Ps L_{p,e}
\]
(actually, these spaces depend only on the points $\Span{p}\in\Ps S^d,\Span{q}\in\Ps S_e$). Moreover, we shall refer to
\[
\left\{\Span{f}\in A_{p,q}:\rk f=2\right\}\;.
\]
as \emph{the rank two locus in $A_{p,q}$}.
\end{defn}

To avoid the \hyperref[c2]{Case~\ref{c2}} that was mentioned in the introduction, note that it occurs exactly when $q$ is a square $v^2$ for some root $\Span{v}\in\Ps{S_1}$ of $p\in S^2$. In the next section we shall find suitable linear forms such that the contraction of $f$ by their product is not a square (a similar caution already appeared in \cite{BD}). This way, we also exclude the occurrence of a singularity in the \hyperref[c11]{Case~\ref{c11}}, which happens if and only if $q=v^2$ with $\Span{v}$ not a root of~$p$. Under that hypothesis, below we determine a suitable parameterization of the rank two locus.

\begin{lemma}\label{Summary}
Let $\dim S_1=2$, $\Span{p}\in S^{d-2}$, with $d\ge 3$, $\Span{q}\in\Ps S_2$, with $q$ not a square, and let $\mathring{R}$ be the rank two locus in $A_{p,q}$ (see \autoref{WLA}).

Given $x_0,x_1\in S_1$ such that $\Span{x_0x_1}=\Span{q}$, there exists a unique \[\mathbf{r}\in\K\left[t^0,t^1\right]_{d-2}\otimes W_{p,q}\subset\mathbf{S}^{d-2}_d\] such that
\[
t^0t^1\cdot\left(\mathbf{r}\restriction_{\left({t^0}^2,{t^1}^2\right)}\right)=p\left(t^0x_0-t^1x_1\right)\left(t^0x_0+t^1x_1\right)^d-p\left(t^0x_0+t^1x_1\right)\left(t^0x_0-t^1x_1\right)^d\;,
\]
and a finite subset $X\subset\Ps^1$ such that
\[
\mathring{R}=\left\{\Span{\mathbf{r}\restriction_{(\lambda,\mu)}}: [\lambda,\mu]\in\Ps^1\setminus X\right\}\;.
\]
\end{lemma}
\begin{proof}
Let $\Span{q}=\Span{x_0x_1}$ with $x_0,x_1\in S_1$. Since $q$ is not a square, $\Span{x_0},\Span{x_1}\in\Ps{S_1}$ are distinct. In the polynomial ring $\mathbf{S}=\K\left[t^0,t^1,x_0,x_1\right]$ we have
\begin{equation}\label{Start}
\Span{t^0t^1q}=\Span{\left(t^0x_0+t^1x_1\right)^2-\left(t^0x_0-t^1x_1\right)^2}\;.
\end{equation}
Let
\[
\mathbf{g}:=p\left(t^0x_0-t^1x_1\right)\left(t^0x_0+t^1x_1\right)^d-p\left(t^0x_0+t^1x_1\right)\left(t^0x_0-t^1x_1\right)^d\;.
\]
Then, \eqref{Start} and the Lebnitz rule for contraction give
\[
p\ins\mathbf{g}\in\K\left[t^0,t^1\right]_{2d-2}\otimes\Span{q}\;,
\]
hence $\mathbf{g}\in\K\left[t^0,t^1\right]_{2d-2}\otimes W_{p,q}\subset\mathbf{S}^{2d-2}_d$. Since
\[
\mathbf{g}\restriction_{\left(0,t^1\right)}=0=\mathbf{g}\restriction_{\left(t^0,0\right)}\;,
\]
$\mathbf{g}$ is divisible by $t^0t^1$. Next, let $\mathbf{h}:=\mathbf{g}/t^0t^1\in\K\left[t^0,t^1\right]_{2d-4}\otimes W_{p,q}$ and note that
\[
\mathbf{h}\restriction_{\left(-t^0,t^1\right)}=\mathbf{h}=\mathbf{h}\restriction_{\left(t^0,-t^1\right)}\;.
\]
Therefore $\mathbf{h}\in\K\left[{t^0}^2,{t^1}^2\right]\otimes W_{p,q}$, and hence $\mathbf{h}=\mathbf{r}\restriction_{\left({t^0}^2,{t^1}^2\right)}$ for some $\mathbf{r}\in\K\left[t^0,t^1\right]_{d-2}\otimes W_{p,q}\subset\mathbf{S}^{d-2}_d$. By definition,
\begin{equation}\label{Erre}
t^0t^1\cdot\left(\mathbf{r}\restriction_{\left({t^0}^2,{t^1}^2\right)}\right)=p\left(t^0x_0-t^1x_1\right)\left(t^0x_0+t^1x_1\right)^d-p\left(t^0x_0+t^1x_1\right)\left(t^0x_0-t^1x_1\right)^d\;,
\end{equation}
and, of course, the above relation uniquely determines $\mathbf{r}$.

Since $\mathbf{r}\in\K\left[t^0,t^1\right]\otimes W_{p,q}$, we have $\mathbf{r}\restriction_{(\lambda,\mu)}\in W_{p,q}$ for all $(\lambda,\mu)\in\K^2$. If $\rho x_0+\theta x_1,\rho x_0-\theta x_1\in S_1$ are not roots of $p$ and $\rho\theta\ne 0$, then $\mathbf{g}\restriction_{(\rho,\theta)}\ne 0$, $\Span{\mathbf{r}\restriction_{(\rho^2,\theta^2)}}=\Span{\mathbf{g}\restriction_{(\rho,\theta)}}$, and $\rk\mathbf{g}\restriction_{(\rho^2,\theta^2)}=2$ (rank one is excluded by~\autoref{Elem}). Since $p\ins\mathbf{g}\restriction_{(\rho,\theta)}\ne 0$, we also have that $\Span{\mathbf{r}\restriction_{(\rho^2,\theta^2)}}$ is not in $\Ps L_{p,2}$ (hence is in $A_{p,q}$). Therefore,
\[
\Span{\mathbf{r}\restriction_{(\rho^2,\theta^2)}}\in\mathring{R}\;.
\]
Conversely, suppose that $\Span{f}\in\mathring{R}$. Then $f={v_0}^d+{v_1}^d$ for some distinct $\Span{v_0},\Span{v_1}\in\Ps S_1$ and $\Span{f}\in A_{p,q}=\Ps W_{p,q}\setminus\Ps L_{p,2}$. But $f\in W_{p,q}$ implies that $p\ins f\in\Span{q}$ and $f\not\in L_{p,2}$ implies that $p\ins f\ne 0$. Since \[p\ins f=\frac{d!}{2}\left(p\left(v_0\right){v_0}^2+p\left(v_1\right){v_1}^2\right)\;,\] and $q$ is not a square, we have that $p\left(v_0\right)$ and $p\left(v_1\right)$ are both nonzero. Then we can rescale $v_0,v_1$ and assume that
\begin{equation}\label{Eu}
f=p\left(v_1\right){v_0}^d-p\left(v_0\right){v_1}^d\;,\qquad p\ins f=\frac{d!}{2}p\left(v_0\right)p\left(v_1\right)\left({v_0}^2-{v_1}^2\right)\;.
\end{equation}
Let $\left(x^0,x^1\right)$ be the basis of $S^1$, dual to $\left(x_0,x_1\right)$. Since $p\ins f\in\Span{q}=\Span{x_0x_1}$, we have ${x^0\,}^2\ins\left({v_0}^2-{v_1}^2\right)=0={x^1\,}^2\ins\left({v_0}^2-{v_1}^2\right)$. Hence
\[
{x^0\left(v_0\right)\,}^2={x^0\left(v_1\right)\,}^2\;,\qquad{x^1\left(v_0\right)\,}^2={x^1\left(v_1\right)\,}^2\;.
\]
Since $v_0\ne v_1$, and up to possibly replace $v_1$ with $-v_1$ when $d$ is even or replace $\left(v_0,v_1\right)$ with $\left(-v_1,v_0\right)$ when $d$ is odd, we deduce
\[
x^0\left(v_0\right)=x^0\left(v_1\right)\;\qquad x^1\left(v_0\right)=-x^1\left(v_1\right)\;.
\]
Since $v_0=x^0\left(v_0\right)x_0+x^1\left(v_0\right)x_1$, $v_1=x^0\left(v_1\right)x_0+x^1\left(v_1\right)x_1=x^0\left(v_0\right)x_0-x^1\left(v_0\right)x_1$, we have that \eqref{Erre} and the first equality in \eqref{Eu} lead to
\[
f=x^0\left(v_0\right)x^1\left(v_0\right)\mathbf{r}\restriction_{\left(x^0\left(v_0\right)^2,x^1\left(v_0\right)^2\right)}\;.
\]
Moreover, we already pointed out that $v_0,v_1$ are not roots of $p$, and $x^0\left(v_0\right)x^1\left(v_0\right)\ne 0$ because $\Span{v_0}\ne\Span{v_1}$.

This way we showed that if
\[
X:=\left\{[\rho^2,\theta^2]\in\Ps^1:p\left(\rho x_0+\theta x_1\right)=0\right\}\cup\{[1,0],[0,1]\}
\]
then we have
\[
\mathring{R}=\left\{\Span{\mathbf{r}\restriction_{(\lambda,\mu)}}: [\lambda,\mu]\in\Ps^1\setminus X\right\}\;.
\]
\end{proof}

\begin{defn}\label{DS}
Throughout this paper, when $\dim S_1=2$, $\Span{p}\in S^{d-2}$ with $d\ge 3$ and $\Span{q}=\Span{x_0x_1}$ with distinct $\Span{x_0},\Span{x_1}\in\Ps S_1$, the notation $\mathbf{r}_{p,x_0,x_1}$ will refer to the polynomial $\mathbf{r}\in\mathbf{S}^{d-2}_d$ determined as in the statement of~\autoref{Summary}.
\end{defn}

For use in later calculations, below we explicitly write down some formulas, whose algebrogeometric meaning is quite elementary.

\begin{lemma}\label{Formulas}
Let $\dim S_1=2$, $p=l^1\cdots l^{d-2}$ with $d\ge 3$ and $\Span{l^1},\ldots ,\Span{l^{d-2}}\in\Ps S^1$, $q=x_0x_1$ with distinct $\Span{x_0},\Span{x_1}\in\Ps S_1$, and let $\mathbf{r}:=\mathbf{r}_{p,x_0,x_1}$ (see \autoref{DS}). For each $i$, let
\[
a^i:=l^i\left(x_0\right)^2t^0-l^i\left(x_1\right)^2t^1\in\K\left[t^0,t^1\right]\;,\qquad\Span{v_i}:=\Span{l^i}^\perp
\]
and $\left[\lambda_i,\mu_i\right]=[l^i\left(x_1\right)^2,l^i\left(x_0\right)^2]\in\Ps^1$ be the root of $a^i$ (that is, $a^i\restriction_{\left(\lambda_i,\mu_i\right)}=0$).

We have:
\begin{itemize}
\item $\forall I\subseteq\{1,\ldots ,d-2\}$, $\prod_{i\in I}a^i$ divides $\left(\prod_{i\in I}l^i\right)\ins\mathbf{r}$;
\item with $\mathbf{r}'_I$ given by $\left(\prod_{i\in I}a^i\right)\mathbf{r}'_I=\left(\prod_{i\in I}l^i\right)\ins\mathbf{r}$, if $I\subsetneq\{1,\ldots ,d-2\}$ then the map \[[\lambda,\mu]\mapsto \Span{\mathbf{r}'_I\restriction_{(\lambda,\mu)}}\] is one-to-one outside a finite subset of $\Ps^1$ {\rm (\footnote{It follows that it is, more precisely, a birational parameterization of a (rational, quasi-projective) curve.})}; moreover, \[\mathbf{r}'_{\{1,\ldots ,d-2\}}=2\,d!q\;;\]
\item given $i,j\in\{1,\ldots ,d-2\}$, we have that $l^j\left(v_i\right)=0\;\Rightarrow\;a^j\restriction_{\left(\lambda_i,\mu_i\right)}=0$;
\item $\forall i\in\{1,\ldots ,d-2\}$, $\mathbf{r}\restriction_{\left(\lambda_i,\mu_i\right)}\in\Span{{v_i}^d}$;
\item if $\mathbf{r}\restriction_{\left(\lambda_i,\mu_i\right)}=0$ then $a^j\restriction_{\left(\lambda_i,\mu_i\right)}=0$ for some $j\ne i$;
\item if $a^j\restriction_{\left(\lambda_i,\mu_i\right)}=0$ for some $j\ne i$ then $\mathbf{r}\restriction_{\left(\lambda_i,\mu_i\right)}=0$ or $l^j\left(v_i\right)=0$.
\end{itemize}
\end{lemma}
\begin{proof}
Note that
\begin{equation}\label{al}
a^i\restriction_{\left({t^0}^2,{t^1}^2\right)}=l^i\left(t^0x_0+t^1x_1\right)l^i\left(t^0x_0-t^1x_1\right)\;.
\end{equation}
With $I\subseteq\{1,\ldots ,d-2\}$, let us set $p_I:=\prod_{i\in I}l^i$, $p'_I:=p/p_I$, $a_I:=\prod_{i\in I}a^i$, $k:=\sharp(I)$ and
\begin{multline}\label{Esse}
\mathbf{s}_I:=p'_I\left(t^0x_0-t^1x_1\right)\left(t^0x_0+t^1x_1\right)^{d-k}-p'_I\left(t^0x_0+t^1x_1\right)\left(t^0x_0-t^1x_1\right)^{d-k}\;.
\end{multline}
By definition of $\mathbf{r}$ and taking into account \eqref{al}, we have
\begin{equation}\label{Ins}
t^0t^1\cdot\left(p_I\ins\mathbf{r}\restriction_{\left({t^0}^2,{t^1}^2\right)}\right)\\
=\frac{d!}{(d-k)!}\,a_I\restriction_{\left({t^0}^2,{t^1}^2\right)}\,\mathbf{s}_I\;.
\end{equation}
Since $\mathbf{s}_I\restriction_{\left(0,t^1\right)}=0=\mathbf{s}_I\restriction_{\left(t^0,0\right)}$, $\mathbf{s}_I$ is divisible by $t^0t^1$. Therefore $p_I\ins\mathbf{r}$ is divisible by~$a_I$. From \eqref{Esse} and \eqref{Ins} also easily follows that if $[\rho^2,\theta^2],\left[{\rho'}^2,{\theta'}^2\right]\in\Ps^1$ are distinct, $\rho\theta\ne 0$, $\rho'\theta'\ne 0$, and $\rho x_0+\theta x_1$, $\rho x_0-\theta x_1$, $\rho' x_0+\theta' x_1$, $\rho' x_0-\theta' x_1$ are not roots of $p'_I$, then the equality
\[
\Span{\mathbf{r}'_I\restriction_{(\rho^2,\theta^2)}}=\Span{\mathbf{r}'_I\restriction_{({\rho'}^2,{\theta'}^2)}}
\]
would imply that the $(d-k)$-th powers of $\rho x_0+\theta x_1$, $\rho x_0-\theta x_1$, $\rho' x_0+\theta' x_1$, $\rho' x_0-\theta' x_1$ are linearly dependent. But this is impossible when $I\subsetneq\{1,\ldots ,d-2\}$, because in that case we have $d-k\ge 3$ and the four linear polynomials are pairwise non-proportional (see \autoref{Elem2}). Therefore the map $[\lambda,\mu]\mapsto \Span{\mathbf{r}'_I\restriction_{(\lambda,\mu)}}$ is one-to-one outside a finite subset of $\Ps^1$, when $I\subsetneq\{1,\ldots ,d-2\}$. Moreover, from \eqref{Ins} easily follows that $\mathbf{r}'_{\{1,\ldots ,d-2\}}=2\,d!q$.

Let us prove the remaining statements (the ones about $(\lambda_i,\mu_i)$, $v_i$). By definition, $\Span{v_i}=\Span{l^i}^\perp$ for all $i$ (that is, $v_i$ is a nonzero root, unique up to a scalar factor, of $l^1$), and the statements are independent of the choice of the representatives $(\lambda_i,\mu_i)$ of $[\lambda_i,\mu_i]$ and $v_i$ of $\Span{l^i}^\perp$. Therefore we can set,  for each $i$, $\rho_i:=l^i(x_1)$, $\theta_i:=-l^i(x_0)$ and assume
\[
v_i=\rho_i x_0+\theta_i x_1\;,\qquad\left(\lambda_i,\mu_i\right)=\left({\rho_i}^2,{\theta_i}^2\right)\;.
\]
Let us also set $v'_i:=\rho_ix_0-\theta_ix_1$. From \eqref{al} we get
\begin{equation}\label{rt}
a^j\restriction_{\left({\rho_i}^2,{\theta_i}^2\right)}=l^j\left(v_i\right)l^j\left(v'_i\right)\;,\qquad\forall i,j\;.
\end{equation}
This immediately gives, for each $i,j$, the implication $l^j\left(v_i\right)=0\;\Rightarrow\;a^j\restriction_{\left(\lambda_i,\mu_i\right)}=0$. 

Taking $I=\emptyset$ in \eqref{Ins} (in other words, writing down the defining relation of $\mathbf{r}$), we get 
\[
\rho_i\theta_i\mathbf{r}\restriction_{\left({\rho_i}^2,{\theta_i}^2\right)}=\mathbf{s}_\emptyset\restriction_{\left(\rho_i,\theta_i\right)}=p\left(v'_i\right){v_i}^d-p\left(v_i\right){v'_i}^d\;.
\]
Note that $p\left(v_i\right)=0$ for all $i$,  because $\Span{v_i}=\Span{l^i}^\perp$ and $l^i$ is a factor of $p$. Therefore, if $\rho_i\theta_i\ne 0$ then
\[
\mathbf{r}\restriction_{\left({\rho_i}^2,{\theta_i}^2\right)}=\frac{p\left(v'_i\right)}{\rho_i\theta_i}{v_i}^d\in\Span{{v_i}^d}
\]
as required and, moreover,  $\mathbf{r}\restriction_{\left({\rho_i}^2,{\theta_i}^2\right)}=0$ if and only if $p\left(v'_i\right)=0$. But $\rho_i\theta_i\ne 0$ also imply $\Span{v'_i}\ne\Span{v_i}$, hence $p\left(v'_i\right)=0$ if and only if $l^j\left(v'_i\right)=0$ for some $j\ne i$. From \eqref{rt} we easily deduce that, in the case $\rho_i\theta_i\ne 0$, the two remaining statements to be proved are true.

Let us now assume $\rho_i=0$ and let $\mathbf{s}':=\mathbf{s}_\emptyset/t^0$. Since $\rho_i=0$, we have $v_i=\theta_ix_1$ and $l^i\left(x_1\right)=0$. Hence
\begin{multline}\label{sp}
\mathbf{s}'=\\
l^i\left(x_0\right)p'_{\{i\}}\left(t^0x_0-t^1x_1\right)\left(t^0x_0+t^1x_1\right)^d-l^i\left(x_0\right)p'_{\{i\}}\left(t^0x_0+t^1x_1\right)\left(t^0x_0-t^1x_1\right)^d\;,
\end{multline}
which gives $\mathbf{s}'\restriction_{(0,\theta_i)}\in\Span{{x_1}^d}=\Span{{v_i}^d}$. From \eqref{Ins} and $\theta_i\ne 0$ ($v_i\ne 0$) we deduce
\[
\mathbf{r}\restriction_{\left({\rho_i}^2,{\theta_i}^2\right)}=\frac{\mathbf{s}'\restriction_{(0,\theta_i)}}{\theta_i}\in\Span{{v_i}^d}\;,
\]
as it was to show. Moreover, $\mathbf{r}\restriction_{\left({\rho_i}^2,{\theta_i}^2\right)}=0$ if and only if $\mathbf{s}'\restriction_{(0,\theta_i)}=0$. But in view of \eqref{sp} and $l^i\left(x_0\right)\ne 0$, we have that $\mathbf{s}'\restriction_{(0,\theta_i)}=0$ if and only if $p'_{\{i\}}\left(v_i\right)=0$, that is, $l^j\left(v_i\right)=0$ for some $j\ne i$. To conclude, it suffices to note that $l^j\left(v_i\right)=0$ if and only if $a^j\restriction_{\left({\rho_i}^2,{\theta_i}^2\right)}=0$, because of \eqref{rt} and $\rho=0$.

The case $\theta_i=0$ can obviously be handled in the same way.
\end{proof}

Next results provide us with sufficient conditions to avoid that too much binary quintics of high rank arise.

\begin{rem}
Let $\dim S_1=2$ and $\Ps W$ be a projective plane in $\Ps S_5$ (i.e., $\dim W=3$). According to \cite[Proposition~4.1]{BD} (\footnote{There is probably a mistake in the proof of that proposition given in \cite{BD}, but to write down a completely correct and detailed proof does not take long.}), there exists a nonempty (Zariski) open subset $U$ of $\Ps W$ such that $\rk f\le 4$ for all $\Span{f}\in U$. For some special $\Ps W$, that result can not be improved, in the sense that for no nonempty open subset $U$ of $\Ps W$ we can have $\rk f\le 3$ for all $\Span{f}\in U$. Indeed, let us take nonzero $x,y\in S^1$, and set $W:=S_5\cap\Ker\con_{x^2y}$, $\Span{u}:=\Span{x}^\perp$, $\Span{v}:=\Span{y}^\perp$. We show that $\rk f\ge 4$ for all $f\in W\setminus\Span{u^5,v^5}$ (which implies nonexistence of $U$). Let $f\in W\setminus\Span{u^5,v^5}$ and $I_f\subset S^\bullet$ be its apolar ideal, that is,
\[
I_f:=\left\{h\in S^\bullet: \con_h(f)=0\right\}=\bigoplus_{d}\Ker f_{d,5-d}\;.
\]
According to \cite[Theorem~1.44(iv)]{IK}, $I_f$ is generated by a form $h\in S^s$ and a form $h'\in S^{7-s}$, with $s\le 3$. But $x^2y\in I_f$, hence $h$ divides $x^2y$ (because $\deg h'=7-s\ge 4)$). Now, if $h$ is squarefree then \cite[Lemma~1.31]{IK} gives $f\in\Span{u^5,v^5}$, which is excluded. Henceforth, the same lemma gives $\rk f=7-s\ge 4$ (\footnote{It would be easy to exhibit, pursuing the same arguments, the whole rank stratification on~$\Ps W$.}).
\end{rem}

It is well-known that the rank of a generic form of degree $d=2s$ or $d=2s+1$ is $s+1$, that is, there exists a nonempty open subset of $\Ps S_d$ such that $\rk f=s+1$ for all $\Span{f}\in U$. Actually, we have a bit more: the set of all $\Span{f}\in\Ps S_d$ with $\rk f=s+1$ is open (and nonempty). This fact is probably widely known as well, but we prefer to give a precise explanation, because we lack a reference.

\begin{rem}\label{Generic}
Let $C_d$ be the curve given by $d$-th powers in $\Ps S_d$, with $\dim S_1=2$. Let $\sigma_r\left(C_d\right)$, $r\le\left\lfloor\frac{d+1}2\right\rfloor$, be the $r$-th secant variety of $C_d$ (see, e.g., \cite[Definition~5.1.1.2]{L}). Let us recall that $\sigma_r\left(C_d\right)=\sigma_r\left(C_d\right)^{\mathrm{lo}}\cup \sigma_r\left(C_d\right)^{\mathrm{hi}}$, with
\[
\sigma_r\left(C_d\right)^{\mathrm{lo}}:=\{P\in\Ps S_d: \rk P\le r\}\;,\qquad\sigma_r\left(C_d\right)^{\mathrm{hi}}:=\{P\in\Ps S_d: \rk P\ge d+2-r\}\;,
\]
by the Comas-Seiguer theorem (see \cite[Theorem~9.2.2.1]{L}).

When $d$ is even, $d=2s$, it immediately follows that the set $U$ of all $\Span{f}\in\Ps S_d$ with $\rk f=s+1$ is the complement of $\sigma_s\left(C_d\right)$, which is a projective variety, hence a (Zariski) closed subset. Therefore $U$ is open.

When $d$ is odd, $d=2s+1$, then the set to be proved being closed is $X:=\sigma^s\left(C_d\right)\cup\sigma^{s+1}\left(C_d\right)^{\mathrm{hi}}$. A proof that $X$ is a projective variety may go as that one for the secant varieties (we do not need to prove irreducibility which, nevertheless, holds as well). One may look at the incidence variety $V\subset\Ps S^{s+1}\times\Ps S_d$, $V:=\left\{\left(\Span{h},\Span{f}\right):h\ins f=0\right\}$ (in geometric terms, the condition prescribes that $\Span{f}$ lie in the subspace spanned by the subscheme $\nu^{s+1}\left(Z\right)$, with $Z\subset\Ps S_1$ being given by $h=0$ and $\nu^{s+1}:\Ps S^1\to\Ps S^{s+1}$ being the Veronese embedding). Now, \cite[Lemma~1.31 and Theorem~1.44(i, iv)]{IK} imply that $\Span{f}\in X$ if and only if $h\ins f=0$ for some nonzero $h\in S^{s+1}$ that is not squarefree. But for nonzero binary forms squarefree means nonsigular;  hence, if $Y\subseteq\Ps S^{s+1}$ is the locus given by the discriminant, which is a projective variety, we have $X=\pi_2\left(\pi_1^{-1}(Y)\right)$, with $\pi_1,\pi_2$ being the projections of $\Ps S^{s+1}\times\Ps S_d$. To conclude, it suffices to recall the basic algebrogeometric result that the image of a projective variety through a morphism is a projective variety as well.
\end{rem}

\begin{lemma}\label{3}
Let $\dim S_1=2$, $\Span{p}\in\Ps S^2$, $\Span{t}\in\Ps S_3$. If $t$ is not a cube then there exists a nonempty open subset $U$ of $\Ps W_{p,t}$ (see \autoref{WLA}) such that $\rk f=3$ for all $\Span{f}\in U$.
\end{lemma}
\begin{proof}
In view of the above \autoref{Generic}, to find a rank three $f\in W_{p,t}$ will suffice.

Let $K:=\Span{t}^\perp=\Ker t_{3,0}$ ($\dim K=3$). Since $t$ is not a cube, for no $x\in S^1$ we can have $K=xS^2$ (in geometric terms, the linear series on $\Ps S_1$ given by $\Ps K$ is without fixed points). Moreover, we have $\Ker t_{2,1}=\Span{q}$ for some $\Span{q}\in\Ps S^2$. Therefore we can find distinct $\Span{x^1},\Span{x^2},\Span{x^3}\in\Ps S^1$ such that $x^1x^2x^3\in K$ and, moreover, $x^1$, $x^2$, $x^3$ do not divide $p$ nor $q$. By dimension reasons, we have $W_{p,t}\cap\Ker\con_{x^1x^2x^3}=\Span{f}$ for some $\Span{f}\in\Ps W_{p,t}$ (the intersection is nonzero because $x^1x^2x^3\in K$, and the sum is  $\Span{x^1x^2x^3p}^\perp$ because $x^1x^2x^3$ and $p$ are coprime). Since $x^1x^2x^3\ins f=0$, \cite[Lemma~1.31]{IK} gives $f\in\Span{{v_1}^5,{v_2}^5,{v_3}^5}$, where $\Span{v_i}:=\Span{x^i}^\perp$ for each $i$; hence $\rk f\le 3$.

To exclude that $\rk f\le 2$, note that in this case we have $q'\ins f=0$ for some $\Span{q'}\in\Ps S^2$. Note also that $S_5\cap\Ker\con_{p}\cap\Ker\con_{x^1x^2x^3}=\{0\}$, otherwise $S_5\cap\Ker\con_{p}$ and $S_5\cap\Ker\con_{x^1x^2x^3}$ would be both contained in a subspace of dimension four, necessarily of the form $\Span{h,h'}^\perp$ with linearly independent $h,h'\in S^5$. This would mean that the coprime forms $x^1x^2x^3\in S^3$ and $p\in S^2$ divide both $h$ and $h'$, which is impossible (cf.\ also \autoref{Elem2}). Now, $q'\ins f=0$ and $S_5\cap\Ker\con_{p}\cap\Ker\con_{x^1x^2x^3}=\{0\}$ imply that $q'\ins t=0$, hence $\Span{q'}=\Span{q}$. This would lead to $f\in\Ker\con_{q}\cap\Ker\con_{x^1x^2x^3}$, and hence $S_5\cap\Ker\con_{q}\cap\Ker\con_{x^1x^2x^3}\ne\{0\}$, which can be excluded as before.
\end{proof}

\begin{lemma}\label{Ka}
Let $\dim S_1=2$, $\Span{x}\in\Ps S^1$, $\Span{f}\in\Ps S_4$, $D:=S_5\cap{\con_x}^{-1}\left(\Span{f}\right)$. If for infinitely many $\Span{g}\in\Ps D$ we have $\rk g\ge 4$, then there exists $\Span{y}\in\Ps S^1$ such that
\[
x^2y\ins g=0\;\quad\forall g\in D
\]
or
\[
xy^2\ins g=0\;\quad\forall g\in D\;.
\]
\end{lemma}
\begin{proof}
If $\rk f_{3,1}=1$, then $y\ins f=0$ for some nonzero $y\in S^1$. Hence $xy\ins g=0$ for all $g\in D$, and the properties to be proved are both true. Therefore, we can assume that $K:=\Ker f_{3,1}$ is two-dimensional.

If for infinitely many $\Span{k}\in\Ps K$ we have that $k$ is divisible by a square, then all $k\in K$ are divisible by a fixed square $y^2$, with $y\in S^1$ (\footnote{It is an easy case of Bertini's theorem (see also \cite[Lemma~1.1, Remark~1.1.1]{K}). Actually, it would also be easy to show that if $k$ is divisible by a square for at least nine $\Span{k}\in\Ps K$ then all $k\in K$ are divisible by a fixed square.}). Hence $K=y^2S^1$, that implies $y^2\ins f=0$ and therefore $xy^2\ins g=0$ for all $g\in D$. Thus, we can assume that $k$ is divisible by a square only for a finite number of $\Span{k}\in\Ps K$.

From \cite[Lemma~1.31 and Theorem~1.44(i,~iv)]{IK} follows that if $g\in S_5$ and $\rk g\ge 4$, then $z^1{z^2\,}^2\ins g=0$ for some nonzero $z^1,z^2\in S^1$ (it is an instance of a fact we already pointed out in \autoref{Generic}). Hence, for each of the infinitely many $\Span{g}\in\Span{D}$ with $\rk g\ge 4$, we can choose $z^1,z^2$ such that $z^1{z^2\,}^2\ins g=0$. If $g\not\in\operatorname{Ker}\con_x$ then we also have $z^1{z^2\,}^2\in K$. But $\operatorname{Ker}\con_x\cap S_5$ consists of exactly one point $\Span{v^5}$ (with $\Span{v}=\Span{x}^\perp$), and $k$ is divisible by a square only for a finite number of $\Span{k}\in\Ps K$. Therefore there exist two distinct points of $\Span{D}$ that give the same $\Span{z^1{z^2\,}^2}\in\Ps K$. Since $\dim D=2$ (because the restriction $S_5\to S_4$ of $\con_x$ is surjective and has a one-dimensional kernel), this implies that $z^1{z^2\,}^2\ins g=0$ for all $g\in D$. In particular, $z^1{z^2\,}^2\ins v^5=0$, which means that $z^1{z^2\,}^2$ is divisible by $x$, and this proves our statement (with $\Span{y}=\Span{z^2}$ if $\Span{x}=\Span{z^1}$ or $\Span{y}=\Span{z^1}$ if $\Span{x}=\Span{z^2}$).
\end{proof}

The above result is closely related with some nice and more general geometric facts, which we think are worthy of being quickly outlined. Indeed, $\Ps D$ is a line in $\Ps S_5$ that meets the rational normal curve $C_5$ in $\Span{v^5}$ (at least). We can generalize the result by dropping this hypothesis on the position of the line with respect to~$C_5$. What matters is that infinitely many points in $\Ps D$ lie on a plane spanned by a divisor of the type $2P+Q$ on $C_5$ (in the scheme-theoretic sense, and with $P,Q$ possibly coinciding). Taking into account \autoref{Generic}, we have that \emph{every} point in $\Ps D$ lies on a plane spanned by a divisor of the type $2P+Q$. We want to show that $\Ps D$ is contained in one of those planes (in the situation of \autoref{Ka}, it easily follows that $\Span{v^5}$ must coincide with $P$ or with~$Q$). Let us suppose the contrary, and then note that `the divisor $2P+Q$ must move' and that the projection of $C_5$ from the line $\Ps D$ is a curve $C'$ in a three-dimensional projective space, for which each divisor $2P+Q$ becomes aligned. The divisor $2P$ can not be fixed, otherwise (since $Q$ moves) $\Ps D$ would be the tangent to $C_5$ in $P$ and therefore contained in \emph{each} of the planes. We also have $Q\ne P$ for a generic choice of the divisor, otherwise $C'$ would be a line (by a well-known result in characteristic zero). We conclude that $C'$ is a space curve, not contained in a plane, such that the generic tangent meets it in another point. When $C'$ is nonsingular, that is exactly what a relevant result of algebraic geometry excludes: see \cite[Theorem~3.1]{Ka} (\footnote{I thank Edoardo Ballico for suggesting the reference \cite{Ka} and correcting my initial overlooking of the regularity hypothesis in the statement of the theorem.}).

Since the results in \cite{Ka} are quite deep, and leave out the case when $C'$ is cuspidal (see also \cite[Remark~3.8]{Ka}), we also point out that when the degree of $C'$ is at most five (as in the situation of \autoref{Ka}), they can be proved by using more elementary considerations, and the restriction on $C'$ can be removed. We quickly outline the proof using an informal, but not uncommon language. First, note that the point $Q$ as well can not be fixed, otherwise the projection from $Q$ would be inseparable on $C'$ (which is impossible since $\operatorname{char}\,\K=0$). Now, if we move a generic divisor $2P+Q$ in its first order infinitesimal neighborhood, then we get a plane containing $3P+2Q$. But this plane must meet the curve in the further point of intersection of the tangent in $Q$. This point, for a generic $2P+Q$, is distinct from $P$, otherwise the generic tangent would be a bitangent (and this is excluded by another known result of projective differential geometry). Since our curve is of degree at most five, to note that it can not intersect a plane in a degree six divisor suffices. To make such informal considerations rigorous is routine when the base field $\K$ is the complex field. For an arbitrary $\K$ of characteristic zero, one might easily use the techniques developed in~\cite{DI}.

\begin{lemma}\label{Lr}
Let $\dim S_1=2$, $\Span{q}\in\Ps S_2$, $\Span{x^1x^2x^3}\in\Ps S^3$, with $x^1,x^2,x^3\in S^1$, $A:=A_{x^1x^2,q}$ (see \autoref{WLA}) and
\[
E:=\left\{\Span{f}\in A: \rk g\ge 4\text{ for infinitely many }\Span{g}\in\Ps\left(S_5\cap{\con_{x^3}}^{-1}\left(\Span{f}\right)\right)\right\}\;.
\]
Then there exists a finite set $F$ such that for each $\Span{f}\in E$ we have  $\Span{\con_{x^1}f}\in F$ or $\Span{\con_{x^2}f}\in F$.
\end{lemma}
\begin{proof}
Let us decompose $q=x_0x_1$ and let
\begin{equation}\label{U}
u_0:=x^3\left(x_1\right)x_0+x^3\left(x_0\right)x_1\;,\qquad u_1:=x^3\left(x_1\right)x_0-x^3\left(x_0\right)x_1\;,
\end{equation}
\[
v_1:=x^1\left(u_1\right){u_0}^3-x^1\left(u_0\right){u_1}^3\;,\qquad v_2:=x^2\left(u_1\right){u_0}^3-x^2\left(u_0\right){u_1}^3\;,
\]
\[
v^h_k:=x^h\left(x_{1-k}\right){x_k}^3-3x^h\left(x_k\right){x_k}^2x_{1-k}\;,\qquad\forall h\in\{1,2\},k\in\{0,1\}\;.
\]
Note that
\begin{equation}\label{vhk}
\con_{x^h}v_h,\,\con_{x^h}v^h_k,\in\Span{x_0x_1}=\Span{q}\;,\qquad\forall h\in\{1,2\},k\in\{0,1\}\;.
\end{equation}
Then, let us define
\[
F:=\left\{\Span{{x_0}^3}\right\}\;,\qquad\text{if }\Span{x_0}=\Span{x_1}
\]
and
\[
F:=\left\{\Span{v_1},\Span{v_2},\Span{v^1_0},\Span{v^1_1},\Span{v^2_0},\Span{v^2_1}\right\}\;,\qquad\text{if }\Span{x_0}\ne\Span{x_1}\;.
\]
Let $\Span{f}\in E$ and $D:={S_5\cap\con_{x_3}}^{-1}\left(\Span{f}\right)$. According to~\autoref{Ka}, there exists $\Span{z^3}\in\Ps S^1$ such that ${x^3\,}^2z^3\ins g=0$ for all $g\in D$,  or $x^3{z^3\,}^2\ins g=0$ for all $g\in D$. Therefore,
\begin{equation}\label{Or}
x^3z^3\ins f=0\qquad\text{or}\qquad {z^3\,}^2\ins f=0\;.
\end{equation}
If $\Span{z^3}\in\left\{\Span{x^1},\Span{x^2}\right\}$, let $i\in\{1,2\}$ be such that $\Span{x^i}=\Span{z^3}$. In the case when $\Span{z^3}\not\in\left\{\Span{x^1},\Span{x^2}\right\}$ but $\Span{x^3}\in\left\{\Span{x^1},\Span{x^2}\right\}$, let $i\in\{1,2\}$ be such that $\Span{x^i}=\Span{x^3}$. Otherwise, let us fix $i\in\{1,2\}$ at leisure. Let us set \[f':=\Span{\con_{x^i}f}\] and denote by $j\in\{1,2\}$ the index other than $i$ (for short, $j:=3-i$). To prove that $\Span{f'}\in F$ will suffice.

Recall that $\Span{f}\in A=W_{x^1x^2,q}\setminus L_{x^1x^2,2}$ (see \autoref{WLA}), and $f\in W_{x^1x^2,q}$ implies that $\con_{x^j}f'\in\Span{q}$ and $f\not\in L_{x^1x^2,2}$ implies that $\con_{x^j}f'\ne 0$. Hence \[\Span{\con_{x^j}f'}=\Span{q}\;.\] We also point out that if $l\ins f'=0$ for some nonzero $l\in S^1$, then the required statement $\Span{f'}\in F$ follows. Indeed, in that case we necessarily have $f'=w^3$ for some nonzero $w\in S_1$. Hence $\con_{x^j}f'\in\Span{w^2}$ and we know that $\Span{\con_{x^j}f'}=\Span{q}$. Thus $\Span{w^2}=\Span{q}$, which leads to $\Span{w}=\Span{x_0}=\Span{x_1}$. Therefore $\Span{w^3}=\Span{{x_0}^3}\in F$ as required.

Suppose first that $\Span{x^i}=\Span{z^3}$. Then, from \eqref{Or} follows that $x^3\ins f'=0$ or $x^i\ins f'=0$ and this leads to $\Span{f'}\in F$, as pointed out above.

Then, we can assume $\Span{x^i}\ne\Span{z^3}$ (henceforth, $\Span{x^j}\ne\Span{z^3}$ as well, by the choice of $i$). If $\Span{x^i}=\Span{x^3}$ and $x^3z^3\ins f=0$, then $z^3\ins f'=0$ and we have again $\Span{f'}\in F$. It follows that to prove the result in the following two cases will suffice:
\begin{itemize}
\item ${z^3\,}^2\ins f'=0$, or
\item $x^3z^3\ins f'=0$ and $\Span{x^i}\ne\Span{x^3}$ (henceforth, $\Span{x^j}\ne\Span{x^3}$ as well).
\end{itemize}
Suppose first that $\Span{x_0}=\Span{x_1}$. In the first case, we have ${z^3\,}^2\ins {x_0}^2=0$, hence $z^3\left(x_0\right)=0$. Therefore $x^jz^3\ins f'=0$. But $\Span{x^j}\ne\Span{z^3}$ and ${z^3\,}^2\ins f'=0$ as well. This gives $z^3\ins f'=0$ and therefore $\Span{f'}\in F$. In the second case, we have $x^3z^3\ins {x_0}^2=0$, hence $x^3\left(x_0\right)=0$ or $z^3\left(x_0\right)=0$. Therefore $x^jx^3\ins f'=0$ or $x^jz^3\ins f'=0$. But $\Span{x^j}$ does not coincide with $\Span{z^3}$, nor with $\Span{x^3}$ (since we are dealing with the second case). Hence $x^3\ins f'=0$ or $z^3\ins f'=0$, and therefore $\Span{f'}\in F$.

Suppose, finally, that $\Span{x_0}\ne\Span{x_1}$. In the first case, we have ${z^3\,}^2\ins q=0$, hence $z^3\left(x_k\right)=0$ for some $k\in\{0,1\}$. If $z^3\left(x_k\right)=0$, then ${z^3\,}^2\ins v^j_k=0$; and $v^j_k\ne 0$ because $\Span{x_0}\ne\Span{x_1}$. Let $\Span{y_j}:=\Span{x^j}^\perp$ and note that $S_3\cap\operatorname{Ker}\con_{x^j}=\Span{{y_j}^3}$. If $\Span{f'}\ne\Span{v^j_k}$, then ${y_j}^3\in\Span{f',v^j_k}$ because $\con_{x^j}f'\in\Span{q}$ and $\con_{x^j}v^j_k\in\Span{q}$ by~\eqref{vhk}. Since ${z^3\,}^2\ins f'={z^3\,}^2\ins v^j_k=0$, from ${y_j}^3\in\Span{f',v^j_k}$ we deduce ${z^3\,}^2\ins {y_j}^3=0$, which is excluded because $\Span{z^3}\ne\Span{x^j}$. Hence $\Span{f'}\ne\Span{v^j_k}$ is excluded, so we get $\Span{f'}=\Span{v^j_k}\in F$ as required. In the second case, we have $x^3z^3\ins x_0x_1=0$. If $x^3\left(x_0\right)=0$ or $x^3\left(x_1\right)=0$, we deduce that $\Span{x^3}=\Span{z^3}$ and we fall again in the first case. Hence we can assume that $x^3\left(x_0\right)\ne 0$ and $x^3\left(x_1\right)\ne 0$. Now, we also have
\[
0=x^3z^3\ins x_0x_1=x^3\left(x_0\right)z^3\left(x_1\right)+x^3\left(x_1\right)z^3\left(x_0\right)\;.
\]
Hence, from \eqref{U} we get $z^3\left(u_0\right)=0$. Since also $x^3\left(u_1\right)=0$, we have that $x^3z^3\ins v_j=0$; and $v_j\ne 0$ since $x^3\left(x_0\right)\ne 0$ and $x^3\left(x_1\right)\ne 0$. As before, from \eqref{vhk} and $\con_{x^j}f'\in\Span{q}$ we deduce that if $\Span{f'}\ne\Span{v_j}$ then $x^3z^3\ins{y_j}^3=0$. But this is excluded, since $\Span{x^3}\ne\Span{x^j}$ and $\Span{z^3}\ne\Span{x^j}$. Thus, we can conclude with $\Span{f'}=\Span{v_j}\in F$.
\end{proof}

\section{On Apolar Configurations of Lines}\label{Lines}

As explained in the introduction, the basic idea we are pursuing to bound the rank of $f$ is to find suitable $l^1,\ldots ,l^k\in S^1$, such that $l^1\cdots l^k\ins f=0$. In more geometric terms, we are dealing with configuration of lines of the plane $\Ps S_1$, that are apolar to the target form $f$. After a brief preparation, soon we shall start finding appropriate configurations.

\begin{lemma}\label{Drop}
Assume $\dim S_1=3$, and recall that we are taking scalars in a field of characteristic zero. If $f\in S_d$ and $V\subset S^1$ is a two-dimensional subspace, then (at least) one of the following statements is true:
\begin{itemize}
\item $l\ins f$ is a $d$-th power for at most two choices of $\Span{l}\in\Ps V$; or
\item there exists $\Span{l}\in\Ps V$ such that $l\ins f=0$.
\end{itemize}
\end{lemma}
\begin{proof}
We assume that the first of the listed statements is false and prove the second. We can choose linearly independent $x^1,x^2$ such that $V=\Span{x^1,x^2}$ and $x^1\ins f={v_1}^d$, $x^2\ins f={v_2}^d$ for some $v_1,v_2\in S_1$. If $v_1=0$ or $v_2=0$ the second statement trivially follows. If not, since there exists one more
\[
\Span{x^3}\in\Span{x^1,x^2}\setminus\left(\Span{x^1}\cup\Span{x^2}\right)
\]
such that $x^3\ins f$ is a $d$-th power, from \autoref{Elem} and linearity on the left of the contraction operation~$\ins$ follows that $\Span{{v_1}^d}=\Span{{v_2}^d}$, and henceforth $x^2\ins f\in\Span{{v_1}^d}$. Then, again by linearity on the left of contraction, we can find a nonzero $l\in\Span{x^1,x^2}=V$ such that $l\ins f=0$.
\end{proof}

The following results help to find appropriate apolar configurations of lines, under appropriate hypotheses.

\begin{rem}\label{Recap}
Let $f\in S_d$, with $\dim S_1=3$, and $\Span{x}\in\Ps S^1$ be such that $x^2\ins f=0$. Suppose that $\Span{p}\in\Ps S^\delta$ is such that $p\ins f=0$ and the curve $p=0$ in $\Ps S_1$ intersects the line $x=0$ in (exactly) $\delta$ distinct points, that henceforth are simple points of the curve. We have that the $\delta$ (distinct) tangents $l^1=0$, $\ldots$, $l^\delta=0$ to the curve in the intersection points are such that $l^1\cdots l^\delta\ins f=0$ (cf.\ the second part of the proof of \cite[Proposition~5.2]{D}).
\end{rem}

\begin{lemma}\label{Double}
Let $f\in S_5$, with $\dim S_1=3$, and suppose that there exist distinct $\Span{x^1},\Span{x^2}\in\Ps S^1$ such that
\[
x^1{x^2\,}^2\ins f=0\;.
\]
Then (at least) one of the following statements is true:
\begin{itemize}
\item there exists a nonzero $l^3\in S^1$ such that $\Span{x^1},\Span{x^2},\Span{l^3}$ are distinct and $x^1x^2l^3\ins f=0$; or
\item there exist nonzero $l^2,l^3,l^4\in S^1$ such that $\Span{x^1},\Span{l^2},\Span{l^3},\Span{l^4}$ are distinct, $x^1l^2l^3l^4\ins f=0$ and
no one of $x^1l^3l^4\ins f$, $x^1l^2l^4\ins f$, $x^1l^2l^3\ins f$ is a square.
\end{itemize}
\end{lemma}
\begin{proof}
Let us set
\[
f':=x^1\ins f\;,\quad G:=x^2\ins f'\;.
\]
Since $x^2\ins G=x^1{x^2\,}^2\ins f=0$, we have $G\in T_\bullet:=\Sy^\bullet\Span{x^2}^\perp=\Ker\partial_{x^2}\subset S_\bullet$. Now, $T_\bullet$ and $T^\bullet:=S^\bullet/\left(x^2\right)$ are rings of binary forms, with an apolarity pairing induced by that between $S^\bullet$ and $S_\bullet$.

Suppose first that $G_{2,1}:T^2\to T_1$ is not surjective (equivalently, $\rk G\le 1$). In this case we have $L^3\ins G=0$ for some nonzero $L^3\in T^1$. But $L^3=l^3+\left(x^2\right)$, and we can choose a representative $l^3\in S^1\setminus\Span{x^1}$ (besides $\not\in\Span{x^2}$). This leads to $x^1x^2l^3\ins f=0$ with distinct $\Span{x^1},\Span{x^2},\Span{l^3}\in\Ps S^1$, as prescribed in the first statement.

Then, let us assume that $G_{2,1}:T^2\to T_1$ is surjective and look at the linear system on $\Ps T_1$ (the line $x^2=0$ in $\Ps S_1$) that is cut by polynomials in $V:=\Span{G}^\perp=\Ker G_{3,0}\subset T^3$. For no $L^3\in T^1\setminus\{0\}$ we can have $V=L^3T^2$, otherwise $L^3\ins G=0$, which is possible only when $G_{2,1}:T^2\to T_1$ is not surjective. Since $\dim V=3$, this means that the linear system cut by $V$ is without fixed points. Therefore we can find distinct $\Span{L^2},\Span{L^3},\Span{L^4}\in\Ps T^1\setminus\left\{\Span{X^1}\right\}$, with $X^1:=x^1+\left(x^2\right)\in T^1$, such that $H:=L^2L^3L^4\in V$ (hence $H\ins G=0$). Since $\dim\Ker G_{2,1}=1$, the linear forms $L^2,L^3,L^4$ can be chosen so that, moreover,
\begin{equation}\label{L}
L^3L^4\ins G\ne 0,\quad L^2L^4\ins G\ne 0,\quad L^2L^3\ins G\ne 0\;.
\end{equation}
Let us fix (at leisure) $x_2\in S_1$ such that $x^2\ins x_2=1$, and set $F:=f'-x_2G$. Since $x^2\ins x_2=1$ and $x^2\ins G=0$, we have
\[
x^2\ins F=\left(x^2\ins f'\right)-G=0\;,
\]
hence $F\in T_4$. Since $G_{2,1}$ is surjective, there exists $K\in T^2$ such that
\[
K\ins G=-H\ins F\;.
\]
Let $h\in S^3$ be the representative of $H\in T^\bullet=S^\bullet/\left(x^2\right)$ that lies in $\Sy^3\Span{x_2}^\perp$ (that is, $h\in\Sy^3\Span{x_2}^\perp$, $H=h+(x^2)$). Let us also choose (at leisure) a representative $k\in S^2$ of $K$ and set $p:=h+x^2k$. Note that $h\ins x_2G=0$, because $h\in\Sy^3\Span{x_2}^\perp$, and recall that $x^2\ins F=0$. We have
\[
p\ins f'=\left(h+x^2k\right)\ins\left(F+x_2G\right)=h\ins F+k\ins G=H\ins F+K\ins G=0\;.
\]
But the curve $p=0$ in $\Ps S_1$ intersects the line $\Ps T_1$ in three distinct points, given by $\Span{L^2}^\perp,\Span{L^3}^\perp,\Span{L^4}^\perp$, because $p+\left(x^2\right)=h+\left(x^2\right)=H=L^2L^3L^4$. According to \autoref{Recap}, we get three tangents $l^2=0$, $l^3=0$, $l^4=0$, such that $x^1l^2l^3l^4\ins f=l^2l^3l^4\ins f'=0$. Of course, we can assume $L_i=l^i+(x^2)$ for each $i\in\left\{2,3,4\right\}$, and since $\Span{L^2},\Span{L^3},\Span{L^4}$ are distinct and lie in $\Ps T^1\setminus\left\{\Span{X^1}\right\}$, we have that $\Span{x^1},\Span{l^2},\Span{l^3},\Span{l^4}$ are distinct.

Finally, suppose that there exist distinct $i,j\in\left\{2,3,4\right\}$ such that $l^il^j\ins f'=u^2$ for some $u\in S_1$. This assumption leads to
\[
L^iL^j\ins G=l^il^j\ins G=x^2\ins u^2=2\left(x^2\ins u\right)u\;.
\]
In view of \eqref{L}, this would imply that $x^2\ins u\ne 0$, which is impossible (because it means that, on one side, $u\not\in\Span{x^2}^\perp=T_1$ and, on the other, $2(x^2\ins u)u=L^iL^j\ins G\in T_1$ is a nonzero multiple of $u$ lying in $T_1$). Hence $x^1l^3l^4\ins f$, $x^1l^2l^4\ins f$ and $x^1l^2l^3\ins f$ are not squares, as prescribed in the second statement.
\end{proof}

We shall also need to particularize the situation of the above lemma, as described in the following remark.

\begin{rem}\label{Double3}
With $f$  and $x^2$ as in \autoref{Double}, suppose that ${x^2\,}^2\ins f=0$, so that we can choose $\Span{x^1}\in\Ps S^1\setminus\left\{\Span{x^2}\right\}$ at leisure. We point out that we can choice $\Span{x^1}$ so that it fulfills the second condition listed in the lemma, or the condition $x^1x^2\ins f=0$. More explicitly:
\begin{itemize}
\item there exists $\Span{x^1}\in\Ps S^1\setminus\{\Span{x^2}\}$ such that $x^1x^2\ins f=0$; or
\item there exist distinct $\Span{x^1},\Span{l^2},\Span{l^3},\Span{l^4}\in\Ps S^1$ such that $x^1l^2l^3l^4\ins f=0$ and no one of $x^1l^3l^4\ins f$, $x^1l^2l^4\ins f$, $x^1l^2l^3\ins f$ is a square.
\end{itemize}
Indeed, let us look at the beginning of the proof of \autoref{Double}, and note that the second condition could be excluded only if $G_{2,1}$ is not surjective, that is, $G$ is not a cube. In the present situation $G$ depends on the choice of $x^1$, hence if $G$ is not a cube for some choice of $x^1$, then the second condition is fulfilled. In the opposite case, it suffices to exploit \autoref{Drop} with $x^2\ins f$ in place of $f$ and with $V$ being whatever two-dimensional subspace of $S^1$ that does not contain $x^2$.
\end{rem}

\begin{lemma}\label{Square}
Let $f\in S_5$, with $\dim S_1=3$. Then (at least) one of the following facts is true:
\begin{itemize}
\item there exist distinct $\Span{l^1},\Span{l^2},\Span{l^3}\in\Ps S^1$ such that $l^1l^2l^3\ins f=0$; or
\item there exist distinct $\Span{l^1},\Span{l^2},\Span{l^3},\Span{l^4}\in\Ps S^1$ such that $l^1l^2l^3l^4\ins f=0$ and no one of $l^1l^3l^4\ins f$, $l^1l^2l^4\ins f$, $l^1l^2l^3\ins f$ is a square.
\end{itemize}
\end{lemma}
\begin{proof}
According to \cite[Proposition~2.7]{BD}, we can find distinct $\Span{x^1}$, $\Span{x^2}$, $\Span{x^3}$, $\Span{x^4}$ in $\Ps S^1$ such that $x^1x^2x^3x^4\ins f=0$. For each $i\in\{1,2,3,4\}$ let
\[
q_i:=\left(\prod_{j\ne i}x^j\right)\ins f\;,
\]
and let us denote by $\nu$ the number of indices $i$ such that $q_i$ is a square. Of course, we can choose $\Span{x^1},\Span{x^2},\Span{x^3},\Span{x^4}$, in such a way that $\nu$ is minimum (among all the choices of distinct $\Span{y^1},\Span{y^2},\Span{y^3},\Span{y^4}\in\Ps S^1$ with $y^1y^2y^3y^4\ins f=0$), and order them so that $q_i$ is a square exactly when $1\le i\le\nu$. If $\nu\le 1$, the second statement is true with $l^i:=x^i$ for all $i$. Therefore, we can assume $\nu\ge 2$. We can also assume that $q_1\ne 0$, because in the opposite case the first statement is true with $\left(l^1,l^2,l^3\right):=\left(x^2,x^3,x^4\right)$.

Since $\nu\ge 2$ and $q_1\ne 0$, we have $q_1={v_1}^2$ for some nonzero $v_1\in S_1$. Let us consider the algebraic curve
\[
C:=\Ps\Span{v_1}^\perp\times\left\{\Span{x^2}\}\times\{\Span{x^3}\}\times\{\Span{x^4}\right\}\subset\Ps S^1\times\Ps S^1\times\Ps S^1\times\Ps S^1\;.
\]
To check that, for each $i$, the set
\[
\left\{\Big(\Span{y^1},\Span{y^2},\Span{y^3},\Span{y^4}\Big): \left(\prod_{j\ne i}y^j\right)\ins f\text{ is a square}\right\}
\]
is algebraic in $\Ps S^1\times\Ps S^1\times\Ps S^1\times\Ps S^1$ is not difficult. It follows that there exists a nonempty (Zariski) open subset $U$ of $C$ such that for all $\left(\Span{l^1},\Span{x^2},\Span{x^3},\Span{x^4}\right)\in U$ and $i$ with $\nu<i\le 4$, we have that
\[
\left(l^1\prod_{j\ne i}x^i\right)\ins f
\]
is not a square. Because of minimality of $\nu$, we also have that $l^1x^3x^4\ins f$ must be a square for all $\Span{l^1}$ different from $\Span{x^3},\Span{x^4}$ and such that \[\left(\Span{l^1},\Span{x^2},\Span{x^3},\Span{x^4}\right)\in U\;.\]
Now let us look at \autoref{Drop}, with $x^3x^4\ins f$ in place of $f$ and $V:=\Span{v_1}^\perp$. In view of the above said, the first statement in that lemma can not occur. Hence, $l^1x^3x^4\ins f=0$ for some $\Span{l^1}\in\Ps\Span{v_1}^\perp$.

If $\Span{l^1}$ is different from $\Span{x^3},\Span{x^4}$, then the first statement in the lemma under proof is true with $l^2:=x^3$, $l^3:=x^4$. If not, then $\Span{l^1}$ coincides with $\Span{x^3}$ or $\Span{x^4}$, and the result follows from \autoref{Double} with $\left(x^4,x^3\right)$ or, respectively, $\left(x^3,x^4\right)$, in place of $\left(x^1,x^2\right)$.
\end{proof}

The above result can easily be refined as follows.

\begin{rem}\label{Cubes}
Let $f\in S_5$, with $\dim S_1=3$. Then (at least) one of the following facts is true:
\begin{itemize}
\item there exist distinct $\Span{l^1},\Span{l^2}\in\Ps S^1$ such that $l^1l^2\ins f=0$; or
\item there exist distinct $\Span{l^1},\Span{l^2},\Span{l^3}\in\Ps S^1$ such that $l^1l^2l^3\ins f=0$ and no one of $l^1l^3\ins f$, $l^1l^2\ins f$ is a cube; or
\item there exist distinct $\Span{l^1},\Span{l^2},\Span{l^3},\Span{l^4}\in\Ps S^1$ such that $l^1l^2l^3l^4\ins f=0$ and no one of $l^1l^3l^4\ins f$, $l^1l^2l^4\ins f$, $l^1l^2l^3\ins f$ is a square.
\end{itemize}
Indeed, if the first statement listed in \autoref{Square} is false, then the second statement is true. But that statement coincides with the third statement here. Then, let us assume that the first statement in \autoref{Square} is true. In this case, it suffices to reiterate the arguments in the proof of \autoref{Square}, with $\Span{l^1},\Span{l^2},\Span{l^3}\in\Ps S^1$ in place of $\Span{x^1},\Span{x^2},\Span{x^3},\Span{x^4}\in\Ps S^1$, and at the end exploit \autoref{Double3} instead of \autoref{Double}.
\end{rem}

\section{Decomposition in a sum of ten fifth powers}

\begin{prop}\label{Decomp}
Let $f\in S_5$, with $\dim S_1=3$. If there exist distinct $\Span{l^1},\ldots,\Span{l^4}\in\Ps S^1$ such that
\begin{itemize}
\item $l^1l^2l^3l^4\ins f=0$,
\item $l^1l^2l^3\ins f$ and $l^1l^2l^4\ins f$ are not squares,
\item $l^1l^3l^4\ins f\ne 0$ and $l^2l^3l^4\ins f\ne 0$,
\end{itemize}
then $\rk f\le 10$.
\end{prop}
\begin{proof}
Let us consider the rings $T_\bullet:=\Sy^\bullet\Span{l^4}^\perp=\Ker\partial_{l^4}$ and $T^\bullet:=S^\bullet/\left(l^4\right)$ with the apolarity pairing induced by that between $S_\bullet$ and $S^\bullet$. Let $p:=l^1l^2l^3+\left(l^4\right)\in T^3$ and note that $q:=l^1l^2l^3\ins f\in T_2$. In view of \autoref{Summary} (with $T$ in place of $S$) and \autoref{DS}, we can consider $\mathbf{r}_4:=\mathbf{r}_{p,x_0,x_1}\in\mathbf{T}:=\K\left[t^0,t^1\right]\otimes T_\bullet$, with $x_0x_1=q$. We have that $\rk\,\mathbf{r}_4\!\restriction_{(\lambda,\mu)}=2$ for all $\left[\lambda,\mu\right]$ outside a finite subset of $\Ps^1$. According to \autoref{Formulas}, we can find $a^1,a^2,a^3\in\Span{t^0,t^1}$ with the properties indicated in that statement.

In the present proof, to turn $a^3$ into $t^0$ by means of a linear change of coordinates will be convenient. To this end, we can certainly find $s^0,s^1$ such that $\Span{s^0,s^1}=\Span{t^0,t^1}$ and $a^3\restriction_{\left(s^0,s^1\right)}=t^0$, and set
\[
\mathbf{f}_4:=\mathbf{r}_4\restriction_{\left(s^0,s^1\right)}\;,\quad a^{i4}:=a^i\restriction_{\left(s^0,s^1\right)}\;,\qquad\forall i\in\{1,2,3\}
\]
(hence $a^{34}=t^0$). Since $T_\bullet$ is a subring of $S_\bullet$, $\mathbf{f}_4$ can be regarded as an element of $\mathbf{S}=\K\left[t^0,t^1\right]\otimes S_\bullet$. Of course, still $\rk\,\mathbf{f}_4\!\restriction_{(\lambda,\mu)}=2$ for all $\left[\lambda,\mu\right]$ outside a finite subset of $\Ps^1$. Moreover, let $\Span{v_{34}}:=\Span{l^3,l^4}^\perp=\Span{l^3+\left(l^4\right)}^\perp$ and, for each $i\in\{1,2\}$, let $\left[\lambda_i,\mu_i\right]\in\Ps^1$ be the root of $a^{i4}$ (that is, $a^{i4}\restriction_{\left(\lambda_i,\mu_i\right)}=0$).

The properties of $\mathbf{r}_4,a^1,a^2,a^3$ stated in \autoref{Formulas} lead (in particular) to:
\begin{itemize}
\item $\forall I\subseteq\{1,2,3\}$, $\prod_{i\in I}a^{i4}$ divides $\left(\prod_{i\in I}l^i\right)\ins\mathbf{f}_4$;
\item with $\mathbf{f}'_{4;I}\in\mathbf{S}$ being defined by setting
\[
\left(\prod_{i\in I}l^i\right)\ins\mathbf{f}_4=\left(\prod_{i\in I}a^{i4}\right)\mathbf{f}'_{4;I}\;,
\]
$\forall I\subsetneq\{1,2,3\}$ the map \[[\lambda,\mu]\mapsto\Span{\mathbf{f}'_{4;I}\restriction_{(\lambda,\mu)}}\]
is one-to-one outside a finite subset of $\Ps^1$; moreover,
\begin{equation}\label{f4}
l^1l^2l^3\ins\mathbf{f}_4=240a^{14}a^{24}t^0q=240a^{14}a^{24}t^0\left(l^1l^2l^3\ins f\right)\;;
\end{equation}
\item for each $j\in\{1,2\}$, $l^j\left(v_{34}\right)=0\;\Rightarrow\;a^{j4}\restriction_{(0,1)}=0$;
\item $\mathbf{f}_4\restriction_{\left(0,1\right)}\in\Span{{v_{34}}^5}$.
\end{itemize}

Now, let us consider $l^3$ in place of $l^4$ and $l^1l^2l^4+\left(l^3\right)$, $l^1l^2l^4\ins f$ in place of $p$, $q$. As before, we get $\mathbf{r}_3\in\mathbf{S}^3_5$, $b^1,b^2,b^4\in\K\left[t^0,t^1\right]$ with the properties listed in \autoref{Formulas} for $\mathbf{r},a^1,a^2,a^3$. Let $\mathbf{r}'_{3;\{1,2\}}\in\mathbf{S}^1_3$ be defined, as in the lemma, by the equality
\[
l^1l^2\ins\mathbf{r}_3=b^1b^{2}\mathbf{r}'_{3;\{1,2\}}\;.
\] 
By one of the properties listed in the lemma, we have
\[
l^4\ins \mathbf{r}'_{3;\{1,2\}}=240b^4\left(l^1l^2l^4\ins f\right)\;,
\]
hence the map
\[
[\lambda,\mu]\mapsto\Span{\mathbf{r}'_{3;\{1,2\}}\restriction_{(\lambda,\mu)}}
\]
(which is one-to-one outside a finite subset of $\Ps^1$) takes values in the projective line $\Ps W_{l^4+(l^3),l^1l^2l^4\ins f}\subset\Ps S_3$ (see \autoref{WLA}). Since $\mathbf{r}'_{3;\{1,2\}}\in\mathbf{S}^1_3$, that map is in fact everywhere defined and, actually, an isomorphism $\Ps^1\overset{\sim}{\longrightarrow}\Ps W_{l^4+(l^3),l^1l^2l^4\ins f}$ of projective lines. Since $\mathbf{f}'_{4;\{1,2\}}$ enjoys the same property (but with a different range line), and
\[
l^4\ins\mathbf{f}'_{4;\{1,2\}}=0\;,\quad l^4\ins\left(l^1l^2\ins f\right)=l^1l^2l^4\ins f\ne 0\;,
\]
the polynomial
\[
240t^0\left(l^1l^2\ins f\right)-\mathbf{f}'_{4;\{1,2\}}\:,
\]
which for reasons that will soon become clear we denote by
\[
\mathbf{f}'_{3;\{1,2\}}\;,
\]
again defines an isomorphism of $\Ps^1$ into a projective line in $\Ps S_3$. But \eqref{f4} implies
\[
l^3\ins\mathbf{f}'_{4;\{1,2\}}=240t^0\left(l^1l^2l^3\ins f\right)\;,
\]
that in turn gives
\[
l^3\ins\mathbf{f}'_{3;\{1,2\}}=0\;.
\] 
Moreover,
\begin{equation}\label{fp3}
l^4\ins\mathbf{f}'_{3;\{1,2\}}=240t^0\left(l^1l^2l^4\ins f\right)\;.
\end{equation}
The last two equalities proves that the range of the map
\[
[\lambda,\mu]\mapsto\Span{\mathbf{f}'_{3;\{1,2\}}\restriction_{(\lambda,\mu)}}
\]
is again $\Ps W_{l^4+(l^3),l^1l^2l^4\ins f}$, as it was for $\mathbf{r}'_{3;\{1,2\}}$. One easily deduce that there exists an invertible linear change on $(t^0,t^1)$ that turns $\mathbf{r}'_{3;\{1,2\}}$ into $\mathbf{f}'_{3;\{1,2\}}$. Performing that change of coordinates on $\mathbf{r}_3,b^1,b^2,b^4$ (\footnote{The change is not intended to act elsewhere; in particular, we keep $\mathbf{f}_4$ unaltered.}), we find $\mathbf{f}_3,a^{13},a^{23},a^{43}$ that enjoy a similar list of properties as $\mathbf{f}_4,a^{14},a^{24},a^{34}$ (note that the notation $\mathbf{f}'_{3;\{1,2\}}$ is coherent with that one in the list, and that \eqref{fp3} implies $a^{43}=t^0$).

Let
\[
\mathbf{g}:=a^{14}a^{24}\mathbf{f}_3+a^{13}a^{23}\mathbf{f}_4\;.
\]
Taking into account that $\mathbf{f}'_{3;\{1,2\}}=240t^0\left(l^1l^2\ins f\right)-\mathbf{f}'_{4;\{1,2\}}$ by definition, we have
\begin{equation}\label{g12}
l^1l^2\ins\mathbf{g}=240a^{13}a^{14}a^{23}a^{24}t^0\left(l^1l^2\ins f\right)\;.
\end{equation}
Let us show that
\begin{equation}\label{Erase}
\mathbf{g}\restriction_{\left(0,1\right)}=0\;.
\end{equation}
Recall that $\mathbf{f}_3\restriction_{\left(0,1\right)},\mathbf{f}_4\restriction_{\left(0,1\right)}\in\Span{{v_{34}}^5}$, henceforth $\mathbf{g}\restriction_{\left(0,1\right)}\in\Span{{v_{34}}^5}$.  If $l^1\ins v_{34}=0$, then $a^{13}\restriction_{\left(0,1\right)}=a^{14}\restriction_{\left(0,1\right)}=0$ by one of the properties in the respective lists, and \eqref{Erase} follows. If $l^2\ins v_{34}=0$ we can argue in the same way. Suppose then that $l^1\ins v_{34}\ne 0$,  $l^2\ins v_{34}\ne 0$. In this case, if $\mathbf{g}\restriction_{\left(0,1\right)}$ had been a nonzero multiple of ${v_{34}}^5$, then $l^1l^2\ins\mathbf{g}\restriction_{\left(0,1\right)}$ would have been a nonzero multiple of ${v_{34}}^3$, which is excluded by \eqref{g12}. Hence \eqref{Erase} is proved.

Now, by \eqref{Erase}, there exists $\mathbf{f}_{34}\in\mathbf{S}^4_5$ such that $t^0\mathbf{f}_{34}=\mathbf{g}$. Let
\[
\mathbf{f}_{12}:=240a^{13}a^{23}a^{14}a^{24}f-\mathbf{f}_{34}\;.
\]
We have
\begin{equation}\label{F1234}
240a^{13}a^{23}a^{14}a^{24}t^0f=t^0\mathbf{f}_{12}+a^{14}a^{24}\mathbf{f}_3+a^{13}a^{23}\mathbf{f}_4
\end{equation}
and from \eqref{g12} we easily deduce
\begin{equation}\label{f12}
l^1l^2\ins\mathbf{f}_{12}=0\;.
\end{equation}
Let us also recall that $a^{34}=t^0=a^{43}$. We have
\[
l^3\ins t^0\mathbf{f}_{34}=l^3\ins\mathbf{g}=a^{14}a^{24}\left(l^3\ins\mathbf{f}_3\right)+a^{13}a^{23}\left(l^3\ins\mathbf{f}_4\right)=a^{13}a^{23}t^0\mathbf{f}'_{4;\{3\}}\;,
\]
hence
\[
l^3\ins\mathbf{f}_{34}=a^{13}a^{23}\mathbf{f}'_{4;\{3\}}\;.
\]
Therefore
\[
l^3\ins\mathbf{f}_{12}=a^{13}a^{23}\left(240a^{14}a^{24}\con_{l^3}f-\mathbf{f}'_{4;\{3\}}\right)\;,
\]
which leads to
\[
l^1l^3\ins\mathbf{f}_{12}=a^{13}a^{23}a^{14}\left(240a^{24}\con_{l^1l^3}f-\mathbf{f}'_{4;\{13\}}\right)\;.
\]
As before, we have that the map $[\lambda,\mu]\mapsto \Span{\mathbf{f}'_{4;\{13\}}\restriction_{(\lambda,\mu)}}$ is an isomorphism of $\Ps^1$ onto a projective line in $\Ps S_3$ (in what follows, to use that it is one-to-one outside a finite subset of $\Ps^1$ would suffice). From
\[
l^4\ins\con_{l^1l^3}f=l^1l^3l^4\ins f\ne 0\;,\qquad l^4\ins\mathbf{f}'_{4;\{13\}}=0
\]
we deduce that the map
\[
\left[\lambda,\mu\right]\mapsto\Span{\left(l^1l^3\ins\mathbf{f}_{12}\right)\restriction_{(\lambda,\mu)}}
\]
is defined outside a finite subset of $\Ps^1$ and one-to-one. Moreover, taking into account \eqref{f12} and possibly enlarging the forbidden finite subset of $\Ps^1$, we can assume that the above map takes values in $A_{l^4+\left(l^2\right),l^1l^3l^4\ins f}$ (see \autoref{WLA}, with the dually paired rings $S^\bullet/\left(l^2\right)$, $\Sy^\bullet\Span{l^2}^\perp=\Ker\partial_{l^2}$ in place of $S^\bullet$, $S_\bullet$). The same argument works for any other $l^il^j\ins\mathbf{f}_{12}$ with $(i,j)\in\{1,2\}\times\{3,4\}$: outside a suitable finite subset of $\Ps^1$, the map
\[
\left[\lambda,\mu\right]\mapsto\Span{\left(l^il^j\ins\mathbf{f}_{12}\right)\restriction_{(\lambda,\mu)}}
\]
is one-to-one and takes values in $A_{l^{7-j}+\left(l^{3-i}\right),l^il^3l^4\ins f}$.

Now, we exploit \autoref{Lr}, with the dually paired rings $S^\bullet/\left(l^1\right)$, $\Sy^\bullet\Span{l^1}^\perp=\Ker\partial_{l^1}$ in place of $S^\bullet$, $S_\bullet$, and with $q=l^2l^3l^4\ins f$, $x^1:=l^3+\left(l^1\right)$, $x^2:=l^4+\left(l^1\right)$, $x^3:=l^2+\left(l^1\right)$. We get sets $E$, $F$, that here we shall denote by $E_1$, $F_1$. In the same way, by exchanging the roles of $l^1$, $l^2$, we get sets $E_2$, $F_2$.

We know that each of $l^3\ins\left(l^2\ins\mathbf{f}_{12}\right)$ and $l^4\ins\left(l^2\ins\mathbf{f}_{12}\right)$ gives a map, respectively into $A_{l^4+\left(l^1\right),l^2l^3l^4\ins f}$ and $A_{l^3+\left(l^1\right),l^2l^3l^4\ins f}$, that is one-to-one outside a finite subset of $\Ps^1$. It follows that outside a (fixed) finite subset of $\Ps^1$, both maps take values outside the finite set $F_1$. According to \autoref{Lr}, this implies that the map into $A_{l^3l^4+\left(l^1\right),l^2l^3l^4\ins f}$ given by $l^2\ins\mathbf{f}_{12}$, outside the mentioned finite subset of $\Ps^1$ takes values outside $E_1$. Similarly, the map into $A_{l^3l^4+\left(l^2\right),l^1l^3l^4\ins f}$ given by $l^1\ins\mathbf{f}_{12}$, outside some finite subset of $\Ps^1$, takes values outside $E_2$.

By the above said, we can fix $[\lambda,\mu]\ne [0,1]$ such that
\[
a^{ij}\restriction_{(\lambda,\mu)}\ne 0\;,\quad\forall (i,j)\in\{1,2\}\times\{3,4\}
\]
and the polynomials
\[
f_{12}:=\mathbf{f}_{12}\restriction_{(\lambda,\mu)}\;,\quad f_3:=\mathbf{f}_3\restriction_{(\lambda,\mu)}\;,\quad f_4:=\mathbf{f}_4\restriction_{(\lambda,\mu)}
\]
(do not vanish and) satisfy the conditions
\[
\Span{l^2\ins f_{12}}\in\Ps S_4\setminus E_1\;,\quad\Span{l^1\ins f_{12}}\in\Ps S_4\setminus E_2\;,\quad\rk f_3=2\;,\quad\rk f_4=2\;.
\]
By \eqref{f12}, $f_{12}$ admits (many) decompositions $f_{12}=f_1+f_2$ with $f_1,f_2\in S_5$, $l^1\ins f_1=0$, $l^2\ins f_2=0$ (and $l^1\ins f_2\ne 0$, $l^2\ins f_1\ne 0$ because $l^1\ins f_{12}\ne 0$ and $l^2\ins f_{12}\ne 0$). If we set $\Span{v_{12}}:=\Span{l^1,l^2}^\perp$, then for every scalar $\nu$, $f_{12}=\left(f_1+\nu{v_{12}}^5\right)+\left(f_2-\nu{v_{12}}^5\right)$ is again a decomposition of the same kind. Therefore, since $\Span{l^2\ins f_{12}}\not\in E_1$, $\Span{l^1\ins f_{12}}\not\in E_2$, we can fix $f_1,f_2$ such that $\rk\Span{f_1}\le 3$, $\Span{f_2}\le 3$.

By evaluating \eqref{F1234} at $(\lambda,\mu)$, we end up with a decomposition
\[
f=g_1+g_2+g_3+g_4
\]
with $\Span{g_i}=\Span{f_i}$ for all $i$ and
\[
\rk f_1\le 3\;,\quad\rk f_2\le 3\;,\quad\rk f_3=2\;,\quad\rk f_4=2\;.
\]
Hence $\rk f\le 10$.
\end{proof}

\begin{prop}\label{Bound}
If $f\in S_5$ and $\dim S_1=3$ then $\rk f\le 10$.
\end{prop}
\begin{proof}
Basically, we argue as in the proofs of \cite[Proposition~3.1]{D} and \cite[Proposition~4.2]{BD}. According to \autoref{Cubes}, there exist $k$ distinct $\Span{l^1},\ldots,\Span{l^k}\in\Ps S^1$, with $2\le k\le 4$, such that $l^1\cdots l^k\ins f=0$, and with additional properties when $k=3$ or $k=4$. Let us set
\[
V_i:=\Sy^5\Span{l^i}^\perp\subseteq S_5\;,
\]
for all $i\in\{1,\ldots ,k\}$, and define
\[
\sigma:\bigoplus_iV_i\to\sum_iV_i\;,\qquad\left(v_1,\ldots ,v_k\right)\mapsto\sum_iv_i\;.
\]
For each $i\in\{1,\ldots ,k\}$, we denote by $\pi_i:\bigoplus_iV_i\to V_i$ the projection map. Moreover, we set $W:=\sigma^{-1}(\Span{f})$ and denote by $\alpha_i$ the restriction $W\to W_i:=\pi_i(W)\subseteq V_i\subseteq\sum_jV_j$ of $\pi_i$. Since $l^1\cdots l^k\ins f=0$, we have $f\in\sum_iV_i$.

The simple case $k=2$ can immediately be worked out as follows. Since $f\in V_1+V_2$, we can pick $v\in W$ such that $\sigma(v)=f$. Therefore $f=\alpha_1(v)+\alpha_2(v)$. Since $\alpha_1(v),\alpha_2(v)$ can be regarded as binary forms, we have $\rk\alpha_1(v),\rk\alpha_2(v)\le 5$, and hence $\rk f\le 5+5=10$.

Suppose now $k=3$. Recall that, moreover,  $l^1l^3\ins f$ and $l^1l^2\ins f$ are not cubes (in particular, they do not vanish). We can also assume that $l^2l^3\ins f\ne 0$, otherwise we fall again in the case $k=2$. The form $t:=l^1l^2\ins f$ belongs to the ring of binary forms $T_\bullet:=\Sy^\bullet\Span{l^3}^\perp=\Ker\partial_{l^3}\subset S_\bullet$. In the dual ring $T^\bullet=S^\bullet/\left(l^3\right)$, let $p:=l^1l^2+\left(l^3\right)$. Note that $W_3=W_{p,t}$, because for each $v\in W$ we have $\alpha_1(v)+\alpha_2(v)+\alpha_3(v)=\sigma(v)\in\Span{f}$ and therefore
\[
l^1l^2\ins\alpha_3(v)=l^1l^2\ins\left(\alpha_1(v)+\alpha_2(v)+\alpha_3(v)\right)\in\Span{t}
\]
(this gives $W_3\subseteq W_{p,t}$; the equality follows, say, by dimension reasons). Since $t$ is not a cube, \autoref{3} gives a nonempty open subset $U_3\subseteq\Ps W_3$ such that $\rk f_3=3$ for all $\Span{f_3}\in U_3$. In the same way, we get a nonempty open subset $U_2\subseteq\Ps W_2$ such that $\rk f_2=3$ for all $\Span{f_2}\in U_2$. Finally, according to \cite[Proposition~4.1]{BD}, there exists a nonempty open subset $U_1\subseteq W_1$ such that $\rk f_1\le 4$ for all $\Span{f_1}\in U_1$.

Now, since the intersection of nonempty (Zariski) open subsets is nonempty, we can pick
\[
\Span{v}\in\left(\Ps\alpha_1\right)^{-1}\left(U_1\right)\cap\left(\Ps\alpha_2\right)^{-1}\left(U_2\right)\cap\left(\Ps\alpha_3\right)^{-1}\left(U_3\right)\;
\]
with $\sigma(v)=f$. We end up with a decomposition $f=\alpha_1(v)+\alpha_2(v)+\alpha_3(v)$, and hence $\rk f\le 4+3+3=10$.

Suppose, finally, that $k=4$. If $l^2l^3l^4\ins f=0$ we fall again in the case $k=3$ (with $l^2,l^3,l^4$ in place of $l^1,l^2,l^3$; note also that if, say, $l^2l^4\ins f$ is a cube, then $l^1l^2l^4\ins f$ is a square). If $l^2l^3l^4\ins f\ne 0$ the result follows from \autoref{Decomp}.
\end{proof}

%


\begin{thebibliography}{14}

\bibitem{AH}
Alexander, James and Hirschowitz, Andr\'e.
\newblock Polynomial interpolation in several variables.
\newblock {\em J. Algebraic Geom.}, \textbf{4(2)}:201--222, 1995.

\bibitem{BD}
Ballico, Edoardo and De Paris, Alessandro.
\newblock {Generic Power Sum Decompositions and Bounds for the Waring rank.}
\newblock arXiv:1312.3494 [math.AG].

\bibitem{BT} Blekherman, Grigoriy and Teitler, Zach.
\newblock {On Maximum, Typical and Generic Ranks.}
\newblock {\em Math. Ann.}, Dec. 2014.
\newblock DOI: 10.1007/s00208-014-1150-3.

\bibitem{D} De Paris, Alessandro.
\newblock {A Proof that the Maximum Rank for Ternary Quartics is Seven.}
\newblock {\em Matematiche (Catania),} to appear. 
\newblock cf.\ arXiv:1309.6475v1 [math.AG].

\bibitem{DI}
De Paris, Alessandro and Ilardi, Giovanna.
\newblock {Some formulae arising in projective-differential geometry.}
\newblock {\em Ann. Univ. Ferrara Sez. VII (N.S.)}, \textbf{47} (2001) 63--88.

\bibitem{DT}
Derksen, Harm and Teitler, Zach.
\newblock {Lower bound for ranks of invariant forms.}
\newblock arXiv:1409.0061 [math.AG]

\bibitem{CCG}
Carlini, Enrico, Catalisano, Maria Virginia and  Geramita, Anthony V.
\newblock The solution to the {W}aring problem for monomials and the sum of
  coprime monomials.
\newblock {\em J. Algebra}, \textbf{370} (2012) 5--14.

\bibitem{CS}
Comas, Gonzalo and Seiguer, Malena.
\newblock {On the rank of a binary form.}
\newblock {\em Found. Comput. Math.}, \textbf{11(1)} (2011) 65--78.

\bibitem{IK}
Iarrobino, Anthony and Kanev, Vassil.
\newblock {\em Power sums, Gorenstein algebras, and determinantal loci.}
\newblock Lecture Notes in Mathematics, vol. 1721, Springer-Verlag, Berlin, 1999.
\newblock Appendix C by Iarrobino and Steven L. Kleiman.

\bibitem{J}
Jelisiejew, Joachim.
\newblock {An upper bound for the Waring rank of a form.}
\newblock {\em Arch. Math. (Basel)} \textbf{102(4)} (2014), 329--336.

\bibitem{K}
Kleppe, Johannes.
\newblock {Representing a Homogenous Polynomial as a Sum of Powers of Linear Forms.}
\newblock Thesis for the degree of Candidatum Scientiarum, Department of Mathematics, Univ. Oslo (1999).

\bibitem{Ka}
Kaji, Hajime.
\newblock {On the tangentially degenerate curves.}
\newblock {\em J. London Math. Soc. (2)}, \textbf{33(3)} (1986) 430--440.

\bibitem{L}
Landsberg, Joseph M.
\newblock {\em {Tensors: Geometry and applications.}}
\newblock {American Mathematical Society (AMS), Providence, RI}, 2012.

\bibitem{S}
Sylvester, James J.
\newblock {On a remarkable discovery in the theory of canonical forms and of hyperdeterminants.}
\newblock {\em Philos. Mag. II} (1851) 391--410.
\newblock Collected Math. Papers, vol. I

\bibitem{T}
Teitler, Zach.
\newblock {Geometric lower bounds for generalized ranks.}
\newblock arXiv:1406.5145 [math.AG].

\end{thebibliography}
\end{document}